\documentclass{article}

\usepackage[english]{babel}
\usepackage{amsmath,amsfonts,amssymb,amsthm}
\usepackage{latexsym}
\usepackage{makeidx}
\newtheorem{deff}{Definition}[section]
\newtheorem{lemma}[deff]{Lemma}
\newtheorem{theorem}[deff]{Theorem}
\newtheorem{corollary}[deff]{Corollary}

\newtheorem{proposition}[deff]{Proposition}
\newtheorem{fact}[deff]{Fact}
\newtheorem{em-example}[deff]{Example}
\newtheorem{em-def}[deff]{Definition}        
\newtheorem{em-remark}[deff]{Remark}         
\newtheorem{em-question}[deff]{Question}

\newtheorem{claim}[deff]{Claim}

\newenvironment{example}{\begin{em-example} \em }{ \end{em-example}}
\newenvironment{remark}{\begin{em-remark} \em }{\end{em-remark}}

\newcommand{\R}{\mathbb R}
\newcommand{\N}{\mathbb N}

\newcommand{\G}{G_\delta}
\newcommand{\Q}{\mathbb Q}
\newcommand{\Z}{\mathbb Z}
\newcommand{\Prm}{\mathbb P}
\newcommand{\T}{\mathbb{T}}
\newcommand{\HG}{\text{Hom}(G,\mathbb T)}
\newcommand{\HGemph}{\text{\emph{Hom}}(G,\mathbb T)}

\input xy
\xyoption{all}

\title{Extremal $\alpha$-pseudocompact abelian groups}
\author{Anna Giordano Bruno}

\date{}

\begin{document}

\maketitle

\begin{abstract}
Let $\alpha$ be an infinite cardinal. Generalizing a recent result of Comfort and van Mill, we prove that every $\alpha$-pseudocompact abelian group of weight $> \alpha$ has some proper dense $\alpha$-pseudocompact subgroup and admits some strictly finer $\alpha$-pseudocompact group topology.
\end{abstract}

\noindent \emph{AMS classification numbers:}  Primary 22B05, 22C05, 40A05; Secondary 43A70, 54A20. \\
\emph{Key words and phrases:} $\alpha$-pseudocompact group, $G_\alpha$-dense subgroup, extremal pseudocompact group, $\alpha$-singular group.

\section{Introduction}

In this paper all topological spaces and groups are supposed to be Tychonov.% and $\alpha$ to be an infinite cardinal.

The following notion was introduced by Kennison.

\begin{deff}\emph{\cite{K}}
Let $\alpha$ be an infinite cardinal. A topological space $X$ is \emph{$\alpha$-pseudocompact} if $f(X)$ is compact in $\R^\alpha$ for every continuous function $f:X\to \R^\alpha$.
\end{deff}

Note that $\omega$-pseudocompactness coincides with pseudocompactness \cite[Theorem 2.1]{K}; so this definition generalizes that of pseudocompact space given by Hewitt \cite{H}. As a direct consequence of the definition, a continuous image of an $\alpha$-pseudocompact space is $\alpha$-pseudocompact. Moreover, an $\alpha$-pseudocompact space of weight $\leq\alpha$ is compact.

If $\alpha\geq \beta$ are infinite cardinals, $\alpha$-pseudocompact implies $\beta$-pseu\-do\-com\-pact and in particular pseudocompact. In \cite[Theorem 1.1]{CR} it is proved that every pseudocompact group $G$ is \emph{precompact} (i.e., the completion $\widetilde G$ is compact \cite{W}).

\bigskip
A pseudocompact group $G$ is \emph{$s$-extremal} if it has no proper dense pseudocompact subgroup and it is \emph{$r$-extremal} if there exists no strictly finer pseudocompact group topology on $G$ \cite{C}.
Recently Comfort and van Mill proved that a pseudocompact abelian group is either $s$- or $r$-extremal if and only if it is metrizable \cite[Theorem 1.1]{CvM2}. The question of whether every either $s$- or $r$-extremal pseudocompact group is metrizable was posed in 1982 \cite{CRob2,CSou} and many papers in the following twenty-five years were devoted to the study of this problem \cite{C,C1,CGal1,CGal2,CGvM,CRob2,CRob,CSou,CvM,CvM2,DGM,Gal2}.

In the survey \cite{C} exposing the story of the solution of this problem there is a wish of extending \cite[Theorem 1.1]{CvM2} to the more general case of non-abelian groups. Moreover it is suggested to consider, for any pair of topological classes $\mathcal P$ and $\mathcal Q$, the problem of understanding whether every topological group $G\in \mathcal P$ admits a dense subgroup and/or a strictly larger group topology in $\mathcal Q$. This problem is completely solved by \cite[Theorem 1.1]{CvM2} in the case $\mathcal P = \mathcal Q = \{\text{pseudocompact abelian groups}\}$. Here we consider and solve the case $\mathcal P=\mathcal Q= \{\text{$\alpha$-pseudocompact abelian groups}\}$.

\medskip
We generalize to $\alpha$-pseudocompact groups the definitions of the different levels of extremality given for pseudocompact groups in \cite{C,CRob,DGM}. For $\alpha=\omega$ we find exactly the definitions of $s$-, $r$-, $d$-, $c$- and weak-extremality.

\begin{deff}
Let $\alpha$ be an infinite cardinal. An $\alpha$-pseudocompact group $G$ is:
\begin{itemize}
	\item \emph{$s_\alpha$-extremal} if it has no proper dense $\alpha$-pseudocompact subgroup;
	\item \emph{$r_\alpha$-extremal} if there exists no strictly finer $\alpha$-pseudocompact group topology on $G$;
	\item \emph{$d_\alpha$-extremal} if $G/D$ is divisible for every dense $\alpha$-pseudocompact subgroup $D$ of $G$;
	\item \emph{$c_\alpha$-extremal} if $r_0(G/D)< 2^\alpha$ for every dense $\alpha$-pseudocompact subgroup $D$ of $G$;
	\item \emph{$\alpha$-extremal} if it is both $d_\alpha$- and $c_\alpha$-extremal.
\end{itemize}
\end{deff}

Moreover we extend for every $\alpha$ the concept of singular group introduced in \cite[Definition 1.2]{DGM} and used later in \cite{DG,DG1}:

\begin{deff}
Let $\alpha$ be an infinite cardinal. A topological abelian group $G$ is \emph{$\alpha$-singular} if there exists a positive integer $m$ such that $w(m G)\leq\alpha$.
\end{deff}

If $\beta\leq\alpha$ are infinite cardinals, then $s_\beta$- (respectively, $r_\beta$-, $d_\beta$-, $c_\beta$-, $\beta$-) extremality yields $s_\alpha$- (respectively, $r_\alpha$-, $d_\alpha$-, $c_\alpha$-, $\alpha$-) extremality, for $\alpha$-pseudocompact groups. Immediate examples of $d_\alpha$- and $c_\alpha$-extremal $\alpha$-pseudocompact gro\-ups are divisible and torsion $\alpha$-pseudocompact groups respectively. Intuitively, $\alpha$-singular groups are those having ``large'' torsion subgroups (see Lemma \ref{alpha-singular}).

In the following diagram we give an idea of the relations among these levels of extremality for $\alpha$-pseudocompact abelian groups. The non-obvious implications in the diagram are proved in Proposition \ref{ccomp}, Theorem \ref{debestr} and Lemma \ref{singular->c-extremal}.

\begin{equation*}\label{extremalities}
\xymatrix{
& w(G)\leq\alpha \ar@{=>}[dr]\ar@{=>}[dl] \ar@{=>}@/^3pc/[ddrr] & & \\
\text{$s_\alpha$-extremal} \ar@{=>}[dr]& & \text{$r_\alpha$-extremal} \ar@{=>}[dl] & \\
& \text{$\alpha$-extremal} \ar@{=>}[dr]\ar@{=>}[dl] & & \text{$\alpha$-singular} \ar@{=>}[dl] \\ 
\text{$d_\alpha$-extremal} & & \text{$c_\alpha$-extremal} &
}
\end{equation*}

The obvious symmetry of this diagram is ``violated'' by the property $\alpha$-singular; but we show that it is equivalent to $c_\alpha$-extremal (see Corollary \ref{alpha-extremal_mega-theorem}).

The main theorem of this paper shows that four of the remaining properties in the diagram coincide:

\begin{theorem}\label{alpha-extremal-solution}
Let $\alpha$ be an infinite cardinal. For an $\alpha$-pseudocompact abelian group $G$ the following conditions are equivalent: 
\begin{itemize}
	\item[(a)]$G$ is $\alpha$-extremal;
	\item[(b)]$G$ is either $s_\alpha$- or $r_\alpha$-extremal;
	\item[(c)]$w(G)\leq\alpha$.
\end{itemize}
\end{theorem}

Our proof of Theorem \ref{alpha-extremal-solution} does not depend on the particular case $\alpha=\omega$, proved in \cite[Theorem 1.1]{CvM2}. However many ideas used here are taken from previous proofs in \cite{CGal2,CGvM,CRob1,CRob,CvM,CvM2,DGM}. Moreover we apply a set-theoretical lemma from \cite{CvM2} (see Lemma \ref{CvM-Lemma}). In each of these cases we give references.

\smallskip
Example \ref{Example-cd} shows that in general $d_\alpha$- and $c_\alpha$-extremality do not coincide with the other levels of extremality.

\medskip
To prove Theorem \ref{alpha-extremal-solution} we generalize a lot of results which hold for pseudocompact abelian groups to $\alpha$-pseudocompact abelian groups. We first establish in \S \ref{alpha-psc} basic properties of $\alpha$-pseudocompact groups.
Then we show that $c_\alpha$-extremal $\alpha$-pseudocompact abelian groups have ``small'' free rank (see Theorem \ref{G>c2}). 
Moreover Theorem \ref{torsion} proves Theorem \ref{alpha-extremal-solution} in the torsion case.

Then we prove that for compact abelian groups $\alpha$-singularity is equivalent to $c_\alpha$-extremality and to a third property of a completely different nature:

\begin{theorem}\label{MegaThm}
Let $\alpha$ be an infinite cardinal. For a compact abelian group $K$ the following conditions are equivalent:
\begin{itemize}
  \item[(a)] $K$ is $c_\alpha$-extremal;
	\item[(b)] $K$ is $\alpha$-singular;
  \item[(c)] there exists no continuous surjective homomorphism of $K$ onto $S^I$, where $S$ is a metrizable compact non-torsion abelian group and $|I|>\alpha$.
\end{itemize}
\end{theorem}

Using (c) we prove that the free rank of non-$\alpha$-singular $\alpha$-pseudocompact abelian groups is ``large'' (see Proposition \ref{non-sing->big_free-rank}). This allows us to extend the equivalence of (a) and (b) to the more general case of $\alpha$-pseudocompact groups:

\begin{corollary}\label{alpha-extremal_mega-theorem}
Let $\alpha$ be an infinite cardinal. An $\alpha$-pseudocompact abelian group is $c_\alpha$-extremal if and only if it is $\alpha$-singular.
\end{corollary}

The last stage of the proof of Theorem \ref{alpha-extremal-solution} is to show that every $\alpha$-singular $\alpha$-extremal group has weight $\leq\alpha$, applying the torsion case of the theorem.

\subsection*{Notation and terminology}

The symbols $\Z$, $\Prm$, $\N$ and  $\N_+$ are used for the set of integers, the set of primes, the set of non-negative integers and the set of positive integers, respectively. The circle group $\T$ is identified with the quotient group $\R/\Z$ of the reals $\R$ and carries its usual compact topology. For a cardinal $\alpha$ and a set $X$ we denote by $X^\alpha$ the product of $\alpha$ many copies of $X$ and by $X^{(\alpha)}$ the direct sum of $\alpha$ many copies of $X$, that is $\bigoplus_\alpha X$.

Let $G$ be an abelian group. We say that $G$ is torsion if every element of $G$ is torsion. Moreover $G$ is torsion-free if no element of $G$ is torsion. Finally, $G$ is non-torsion if there exists at least one element of $G$, which is not torsion.
The group $G$ is said to be \emph{of exponent $n\in\N_+$} if $n$ is such that $n G=0$. Equivalently we say that it is bounded-torsion. If $m$ is a positive integer, $G[m]=\{x\in G:m x=0\}$ and $\Z(m)$ is the cyclic group of order $m$.

We denote by $r_0(G)$ the free rank of $G$ and by $r_p(G)$ the $p$-rank of $G$, for $p\in\Prm$. Moreover $r(G):=r_0(G)+\sup_{p\in\Prm}r_p(G)$ is the rank of $G$.
If $H$ is a group and $h:G \rightarrow H$ is a homomorphism, then we denote by $\Gamma_h:=\{(x,h(x)), \; x\in G\}$ the graph of $h$.

We recall the definitions of some cardinal invariants.
For a topological space $X$ the \emph{weight} $w(X)$ of $X$ is the minimum cardinality of a base for the topology on $X$. % and
%	\item the \emph{density} $d(X)$ of $X$ is the minimal cardinality of a dense subset of $X$.
%\end{itemize}
%For a topological group $G$, we will denote by $\mathcal{V}_{G}(0)$ the filter of the neigh\-bor\-hoods of $e_G$ in $G$. Then:
Moreover, if $x\in X$,
\begin{itemize}
	\item the \emph{character $\chi(x,X)$ at $x$} of $X$ is the minimal cardinality of a basis of the filter of the neigh\-bor\-hoods of $x$ in $X$, and 
	\item the \emph{character} of $X$ is $\chi(X)=\sup_{x\in X}\chi(x,X)$.
\end{itemize}
Analogously
\begin{itemize}
\item the \emph{pseudocharacter $\psi(x,X)$ at $x$} of $X$ is the minimal cardinality of a family $\mathcal{F}$ of neighborhoods of $x$ in $X$ such that $\bigcap_{U\in {\mathcal F}}U=\{x\}$, and
	\item the \emph{pseudocharacter} of $X$ is $\psi(X)=\sup_{x\in X}\psi(x,X)$.
\end{itemize}
In general $\psi(X)\leq\chi(X)\leq w(X)\leq 2^{|X|}\ \ \text{and}\ \ |X|\leq 2^{w(X)}$.
%For a topological space $X$, we denote  by $w(X)$ the \emph{weight} of $X$ (i.e., the minimum cardinality of a base for the topology on $X$). 

The interior of a subset $A$ of $X$ is the union of all open sets within $A$ and is denoted by $\text{Int}_X(A)$ and $\overline{A}^X$ is the closure of $A$ in $X$ (sometimes we write only $\overline A$ when there is no possibility of confusion).

Let $G$ be a topological group. 
%By $\chi(G)$ we denote the \emph{character} of $G$ (that is the minimal cardinality of a basis of the neighborhoods of $e_G$ in $G$) and by $\psi(G)$ the \emph{pseudocharacter} of $G$ (i.e., the minimum size of a family $\mathcal B$ of neighborhoods of $e_G$ in $G$ such that $\bigcap_{U\in \mathcal B}U=\{e_G\}$).
If $M$ is a subset $G$, then $\langle M\rangle$ is the smallest subgroup of $G$ containing $M$. We denote by $\widetilde G$ the two-sided completion of $G$; in case $G$ is precompact $\widetilde G$ coincides with the Weil completion.

For any abelian group $G$ let $\HG$  be the group of all homomorphisms of $G$ to the circle group $\T$. 
When $(G,\tau)$ is an abelian topological group, the set of $\tau$-continuous homomorphisms $\chi:G\to\T$ (\emph{characters}) is a subgroup of $\HG$ and is denoted by $\widehat{G}$; endowed with the compact-open topology, $\widehat G$ is the Pontryagin dual of $G$.

For undefined terms see \cite{DPS,E,HR}.

\section{The $\alpha$-pseudocompactness}\label{alpha-psc}

%We recall that in general $\psi(G)\leq\chi(G)\leq w(G)$ for a topological group $G$. Moreover 
For compact groups we recall the following results about cardinal invariants. (We use the second part of (a) without giving reference, because it is a well-known fact.)

\begin{fact}\label{dens1}\emph{\cite{HR,HR2}}
Let $K$ be a compact group of weight $\geq\omega$. Then:
\begin{itemize}
	\item[(a)]$d(K)=\log w(K)$ and $|K|=2^{w(K)}$;
	\item[(b)]$\psi(K)=\chi(K)=w(K)$;
	\item[(c)]$w(K)=|\widehat K|$.
\end{itemize}
\end{fact}

To begin studying extremal $\alpha$-pseu\-do\-com\-pact groups we need a characterization of $\alpha$-pseu\-do\-com\-pact groups similar to that of pseudocompact groups given by Comfort and Ross theorem, that is Theorem \ref{cr-theorem} below. We find this characterization in Theorem \ref{cr-alpha} combining Theorem \ref{cr-theorem} with the following result.

\begin{deff}\emph{\cite{GST}}
Let $X$ be a topological space and let $\alpha$ be an infinite cardinal.
\begin{itemize}
	\item A \emph{$G_\alpha$-set} of $X$ is the intersection of $\alpha$ many open subsets of $X$.
	\item A subset of $X$ is \emph{$G_\alpha$-dense} in $X$ if it has non-empty intersection with every $G_\alpha$-set of $X$.
\end{itemize}
\end{deff}

The $G_\alpha$-sets for $\alpha=\omega$ are the well known $G_\delta$-sets. A topological space $X$ has $\psi(X)\leq\alpha$ precisely when $\{x\}$ is a $G_\alpha$-set of $X$ for every $x\in X$.

%\begin{deff}\emph{\cite[Definition 1.1]{GST}}
%Let $\alpha$ be an infinite cardinal. A subset $B$ of a topological space $X$ is \emph{$C_\alpha$-compact} in $X$ if $f(B)$ is compact for every continuous function $f:X\to \R^\alpha$.
%\end{deff}
\medskip
The next result is a corollary of \cite[Theorem 1.2]{GST}.
If $X$ is a topological space, then we indicate by $\beta X$ its \v Cech-Stone compactification.

%\begin{theorem}\label{C_alpha-compact}\emph{\cite[Theorem 1.2]{GST}}
%Let $\alpha$ be an infinite cardinal, let $X$ be a topological space and $B$ a subset of $X$. Then the following conditions are equivalent:
%\begin{itemize}
%	\item[(a)]$B$ is $C_\alpha$-compact in $X$;
%	\item[(b)]$f(B)$ is compact for every continuous function $f:X\to Y$, where $Y$ is a topological space of weight $\leq\alpha$;
%	\item[(c)]$B$ is $G_\alpha$-dense in $\overline B^{\beta X}$.
%\end{itemize}
%\end{theorem}

\begin{corollary}\label{C_alpha-compact-cor}
Let $\alpha$ be an infinite cardinal and let $X$ be a topological space. Then $X$ is $\alpha$-pseudocompact if and only if $X$ is $G_\alpha$-dense in ${\beta X}$.
\end{corollary}

The following theorem, due to Comfort and Ross, characterizes pseudocompact groups.

\begin{theorem}\label{cr-theorem}\emph{\cite[Theorem 4.1]{CR}}
Let $G$ be a precompact group. Then the following conditions are equivalent:
\begin{itemize}
\item[(a)]$G$ is pseudocompact;
\item[(b)]$G$ is $\G$-dense in $\widetilde G$;
%\item[(c)]every continuous function $f:G\rightarrow\mathbb{R}$ can be extended to $\widetilde G$;
%\item[(d)]every continuous function $f:G\rightarrow\mathbb{R}$ is uniformly continuous;
\item[(c)]$\widetilde G=\beta G$.
\end{itemize}
\end{theorem}

The next theorem characterizes an $\alpha$-pseu\-do\-com\-pact group in terms of its completion.

\begin{theorem}\label{cr-alpha}
Let $\alpha$ be an infinite cardinal and let $G$ be a precompact group. Then the following conditions are equivalent:
\begin{itemize}
\item[(a)]$G$ is $\alpha$-pseudocompact;
\item[(b)]$G$ is $G_\alpha$-dense in $\widetilde G=\beta G$.
\end{itemize}
\end{theorem}
\begin{proof}
Note that both conditions (a) and (b) imply that $G$ is pseudocompact. So in particular $\widetilde G=\beta G$ by Theorem \ref{cr-theorem}.
Then (a)$\Leftrightarrow$(b) is given precisely by Corollary \ref{C_alpha-compact-cor}.
%(a)$\Rightarrow$(b) Let $\widetilde f:\widetilde G\to\R^\alpha$ be a continuous function. Then $f=\widetilde f\restriction_G:G\to \R^\alpha$ is continuous. Since $G$ is $\alpha$-pseudocompact, it follows that $\widetilde f(G)=f(G)$ is compact and so $G$ is $C_\alpha$-compact in $\widetilde G$. We apply Theorem \ref{C_alpha-compact} to conclude that $G$ is $G_\alpha$-dense in $\widetilde G$.
%
%(b)$\Rightarrow$(a) By Theorem \ref{C_alpha-compact} $G$ is $C_\alpha$-compact in $\widetilde G$. Let $f:G\to \R^\alpha$ be a continuous function. Since $G$ is pseudocompact, Theorem \ref{cr-theorem} implies that there exists a continuous function $\widetilde f:\widetilde G\to \R^\alpha$ such that $\widetilde f\restriction_G=f$. Then $f(G)=\widetilde f(G)$ is compact and so $G$ is $\alpha$-pseudocompact.
\end{proof}

\begin{corollary}\label{alpha-psc<->G_alpha-dense}
Let $\alpha$ be an infinite cardinal. Let $G$ be a topological group and $D$ a dense subgroup of $G$. Then $D$ is $\alpha$-pseudocompact if and only if $D$ is $G_\alpha$-dense in $G$ and $G$ is $\alpha$-pseudocompact.
\end{corollary}
\begin{proof}
Suppose that $D$ is $\alpha$-pseudocompact. It follows that $\widetilde D$ is compact and $D$ is $G_\alpha$-dense in $\widetilde D$ by Theorem \ref{cr-alpha}. Since $D$ is dense in $G$, $\widetilde D=\widetilde G$ and hence $D$ is $G_\alpha$-dense in $G$ and $G$ is $G_\alpha$-dense in $\widetilde G$. By Theorem \ref{cr-alpha} $G$ is $\alpha$-pseudocompact.

Assume that $G$ is $\alpha$-pseudocompact and that $D$ is $G_\alpha$-dense in $G$. Then $G$ is $G_\alpha$-dense in $\widetilde G$ by Theorem \ref{cr-alpha} and $D$ is $G_\alpha$-dense in $\widetilde G$. So $\widetilde G=\widetilde D$ and hence $D$ is $\alpha$-pseudocompact by Theorem \ref{cr-alpha}.
\end{proof}

\begin{lemma}\label{finite_index}
Let $\alpha$ be an infinite cardinal. Let $G$ be a topological group and $H$ an $\alpha$-pseudocompact subgroup of $G$ such that $[G:H]<\infty$. Then $G$ is $\alpha$-pseudocompact.
\end{lemma}
\begin{proof}
It suffices to note that each (of the finitely many) cosets $x H$ is $\alpha$-pseudocompact.
\end{proof}

The following lemma is the generalization to $\alpha$-pseudocompact groups of \cite[Theorem 3.2]{CRob}.

\begin{lemma}\label{psi->w}
Let $\alpha$ be an infinite cardinal. If $G$ is an $\alpha$-pseudocompact group and $\psi(G)\leq\alpha$, then $G$ is compact and so $w(G)=\psi(G)\leq\alpha$.
\end{lemma}
\begin{proof}
Since $\psi(G)\leq\alpha$, it follows that $\{e_G\}=\bigcap_{\lambda<\alpha}U_\lambda$ for neighborhoods $U_\lambda$ of $e_G$ in $G$ and by the regularity of $G$ it is possible to choose every $U_\lambda$ closed in $G$. Let $K=\widetilde G$. Then $\bigcap_{\lambda<\alpha}\overline {U_\lambda}^K$ contains a non-empty $G_\alpha$-set $W$ of $K$. Moreover $G\cap W\subseteq G\cap \bigcap_{\lambda<\alpha}\overline {U_\lambda}^K =\{e_G\}$ and $W\setminus\{e_G\}$ is a $G_\alpha$-set of $K$. Since $G$ is $G_\alpha$-dense in $K$ by Theorem \ref{cr-alpha}, this is possible only if $W=\{e_G\}$. So $\psi(K)\leq\alpha$ and we can conclude that $G=K$ is compact. Moreover $w(G)=\psi(G)$ by Fact \ref{dens1}(b).
\end{proof}

The next proposition covers the implication (c)$\Rightarrow$(b) of Theorem \ref{alpha-extremal-solution}, even for non-necessarily abelian groups.

\begin{proposition}\label{ccomp}
Let $\alpha$ be an infinite cardinal and let $(G,\tau)$ be a compact group of weight $\leq\alpha$. Then $(G,\tau)$ is $s_\alpha$- and $r_\alpha$-extremal.
\end{proposition}
\begin{proof}
First we prove that $(G,\tau)$ is $s_\alpha$-extremal. Let $D$ be a dense $\alpha$-pseudocompact subgroup of $(G,\tau)$. Then $w(D)\leq\alpha$ and so $D$ is compact. So $D$ is closed in $(G,\tau)$ and therefore $D=G$.

Now we prove that $(G,\tau)$ is $r_\alpha$-extremal. Let $\tau'$ be an $\alpha$-pseudocompact group topology on $G$ such that $\tau'\geq\tau$. Since $\psi(G,\tau)\leq\alpha$, it follows that also $\psi(G,\tau')\leq \alpha$.
By Lemma \ref{psi->w} $(G,\tau')$ is compact. Then $\tau'=\tau$.
\end{proof}

\subsection{The family $\Lambda_\alpha(G)$}

Let $G$ be a topological group and $\alpha$ an infinite cardinal. We define $$\Lambda_\alpha(G)=\{N\triangleleft G:N\ \text{closed $G_\alpha$} \}.$$ Usually $\Lambda_\omega(G)$ is denoted by $\Lambda(G)$ (see \cite{CRob,DGM}). For $\alpha\geq\beta$ infinite cardinals $\Lambda_\alpha(G)\supseteq\Lambda_\beta(G)$.

In Theorem \ref{lambda4} we prove that for $\alpha$-pseudocompact groups the families in the following claim coincide.

\begin{claim}\label{obvious-claim}
Let $\alpha$ be an infinite cardinal and let $G$ be a topological group. Then $$\Lambda_\alpha(G)  \supseteq\{N\triangleleft G:\text{closed},\ \psi(G/N)\leq\alpha\} \supseteq\{N\triangleleft G: \text{closed},\ w(G/N)\leq\alpha\}.$$ 
\end{claim}
\begin{proof}
Let $N$ be a closed normal subgroup of $G$ and suppose that $w(G/N)\leq\alpha$. It follows that $\psi(G/N)\leq\alpha$. So $N$ is a $G_\alpha$-set of $G$ and hence $N\in\Lambda_\alpha(G)$.
\end{proof}

For $\alpha=\omega$, the following lemma is \cite[Lemma 1.6]{CRob}.

\begin{lemma}\label{lambda1}
Let $\alpha$ be an infinite cardinal. Let $G$ be a precompact group and $W$ a $G_{\alpha}$-set of $G$ such that $e_G\in W$. Then $W$ contains some $N\in\Lambda_\alpha(G)$ such that $\psi(G/N)\leq\alpha$.
\end{lemma}
\begin{proof}
Let $W=\bigcap_{i\in I}U_i$, where $U_i$ are open subsets of $G$ and $|I|=\alpha$. 
%Since there exists a bijection between $I$ and $I\times \N$, then $W=\bigcap_{i\in I}\bigcap_{n\in\N}U_{i,n}$.
Let $i\in I$. Since $U_i$ is a $\G$-set of $G$ containing $e_G$, then there exists $N_i\in\Lambda(G)$ such that $N_i\subseteq U_i$ and $\psi(G/N_i)\leq\omega$ \cite[Lemma 1.6]{CRob}. Let $N=\bigcap_{i\in I}N_i$. Then $N\in\Lambda_\alpha(G)$. Moreover $\psi(G/N)\leq\alpha$ because there exists a continuous injective homomorphism $G/N\to\prod_{i\in I}G/N_i$ and so $\psi(G/N)\leq\psi(\prod_{i\in I}G/N_i)\leq\alpha$.
\end{proof}

\begin{corollary}\label{lambda2}
Let $\alpha$ be an infinite cardinal. If $G$ is precompact and $W$ is a $G_{\alpha}$-set of $G$, then there exist $a\in W$ and $N\in\Lambda_\alpha(G)$ such that $a N\subseteq W$. So a subset $H$ of $G$ is $G_\alpha$-dense in $G$ if and only if $(x N)\cap H\neq \emptyset$ for every $x\in G$ and $N\in\Lambda_\alpha(G)$.
\end{corollary}

The next theorem and the first statement of its corollary were proved in the case $\alpha=\omega$ in \cite[Theorem 6.1 and Corollary 6.2]{CRob}.

\begin{theorem}\label{lambda4}%\emph{\cite[Teorema 6.1]{comfrobiii}}%6.1
Let $\alpha$ be an infinite cardinal and let $G$ be a precompact group. Then $G$ is $\alpha$-pseudocompact if and only if $w(G/N)\leq\alpha$ for every $N\in\Lambda_\alpha(G)$.
\end{theorem}
\begin{proof}
Suppose that $G$ is $\alpha$-pseudocompact. Since $N\in\Lambda_\alpha(G)$, by Lemma \ref{lambda1} there exists $L\in\Lambda_\alpha(G)$ such that $L\leq N$ and $\psi(G/L)\leq\alpha$. Thanks to Lemma \ref{psi->w} we have that $G/L$ is compact of weight $w(G/L)=\psi(G/L)\leq\alpha$. Since $G/N$ is continuous image of $G/L$, it follows that $w(G/N)\leq\alpha$.

Suppose that $w(G/N)\leq\alpha$ for every $N\in\Lambda_\alpha(G)$. By Corollary \ref{lambda2} and Theorem \ref{cr-alpha} it suffices to prove that $x M\cap G\neq\emptyset$ for every $x\in\widetilde G$ and every $M\in\Lambda_\alpha(\widetilde G)$. Let $\widetilde{\pi}:\widetilde G\to\widetilde G/M$ be the canonical projection, $\pi=\widetilde{\pi}\restriction_G$ and $N=G\cap M$. Hence $N\in\Lambda_\alpha(G)$. By the hypothesis $w(G/N)\leq\alpha$ and so $G/N$ is compact. Since $\pi(G)$ is continuous image of $G/\ker\pi=G/N$, so $\pi(G)$ is compact as well. Since $G$ is dense in $\widetilde G$, it follows that $\pi(G)$ is dense in $\widetilde G/M$ and so $\pi(G)=\widetilde G/M$. Consequently $x M\in\pi(G)=\{g M:g\in G\}$ and hence $x M=g M$ for some $g\in G$, that is $g\in x M\cap G\neq\emptyset$.
\end{proof}

\begin{corollary}\label{lambda8}
Let $\alpha$ be an infinite cardinal. Let $G$ be an $\alpha$-pseudocompact abelian group and let $N\in\Lambda_\alpha(G)$. Then:
\begin{itemize}
	\item[(a)]if $L\in\Lambda_\alpha(N)$, then $L\in\Lambda_\alpha(G)$;
	\item[(b)]$N$ is $\alpha$-pseudocompact;
	\item[(c)]if $L$ is a closed subgroup of $G$ such that $N\subseteq L$, then $L\in\Lambda_\alpha(G)$.
\end{itemize}
\end{corollary}
\begin{proof}
(a) Since $N$ is closed in $G$ and $L$ is closed in $N$, it follows that $L$ is closed in $G$. Moreover $L$ is a $G_\alpha$-set of $G$, because $N$ is a $G_\alpha$-set of $G$ and $L$ is a $G_\alpha$-set of $N$. %Since $L$ is a closed subgroup of $G$, by Claim \ref{obvious-claim} it suffices to prove that $w(G/L)\leq\alpha$. By Theorem \ref{lambda4} $w(G/N)\leq\alpha$ and $w(N/L)\leq\alpha$. Since $G/N$ is algebraically isomorphic to $(G/L)/(N/L)$, it follows that $w(G/L)=w(G/N)\cdot w(N/L)\leq\alpha$, and so $w(G/L)\leq\alpha$.

(b) Let $L\in\Lambda_\alpha(N)$. By (a) $L$ is a $G_\alpha$-set of $G$ and so there exists $M\in\Lambda_\alpha(G)$ such that $M\subseteq L$ by Lemma \ref{lambda1}. By Theorem \ref{lambda4} $w(G/M)\leq\alpha$ and consequently $w(N/M)\leq w(G/M)\leq\alpha$. Since $N/L$ is continuous image of $N/M$, it follows that $w(N/L)\leq w(N/M)\leq\alpha$. Hence $N$ is $\alpha$-pseudocompact by Theorem \ref{lambda4}.

(c) Since $w(G/N)\leq\alpha$ by Theorem \ref{lambda4} and $G/L$ is continuous image of $G/N$, it follows that $w(G/L)\leq\alpha$. Hence $L\in\Lambda_\alpha(G)$ by Claim \ref{obvious-claim}.
\end{proof}

Items (a) and (b) of this corollary and the following lemmas were proved in the pseudocompact case in \cite[Section 2]{DGM}.

\begin{lemma}\label{gdeltad1}
Let $\alpha$ be an infinite cardinal, $G$ an $\alpha$-pseudocompact abelian group and $D$ a subgroup of $G$. Then:
\begin{itemize}
\item[(a)]$D$ is $G_\alpha$-dense in $G$ if and only if $D+N=G$ for every $N\in\Lambda_\alpha(G)$;
\item[(b)]if $D$ is $G_\alpha$-dense in $G$ and $N\in\Lambda_\alpha(G)$, then $D\cap N$ is $G_\alpha$-dense in $N$ and $G/D$ is algebraically isomorphic to $N/(D\cap N)$.
\end{itemize}
\end{lemma}
\begin{proof}
(a) follows from Corollary \ref{lambda2} and (b) follows from (a) and Corollary \ref{lambda8}.
\end{proof}

\begin{lemma}\label{denspsc}
Let $\alpha$ be an infinite cardinal, $G$ an $\alpha$-pseudocompact abelian group, $L$ a closed subgroup of $G$ and $\pi:G\rightarrow G/L$ the canonical projection. 
\begin{itemize}
	\item[(a)]If $N\in\Lambda_\alpha(G)$ then $\pi(N)\in\Lambda_\alpha(G/L)$.
	\item[(b)]If $D$ is a $G_\alpha$-dense subgroup of $G/L$ then $\pi^{-1}(D)$ is a $G_\alpha$-dense subgroup of $G$.
\end{itemize}
\end{lemma}
\begin{proof}
(a) By Theorem \ref{lambda4} we have $w(G/N)\leq\alpha$. Since $(G/L)/\pi(N)= (G/L)/((N+L)/L)$ and $(G/L)/((N+L)/L)$ is topologically isomorphic to $G/(N+L)$, it follows that $w((G/L)/\pi(N))\leq\alpha$. Hence $\pi(N)\in\Lambda_\alpha(G/L)$ by Claim \ref{obvious-claim}.

(b) Let $N\in\Lambda_\alpha(G)$. Since $\pi(N+\pi^{-1}(D))=\pi(N)+D$ and by (a) $\pi(N)\in\Lambda_\alpha(G/L)$, Lemma \ref{gdeltad1}(a) implies that $\pi(N)+D=G/L$. Then $\pi(N+\pi^{-1}(D))=G/L$ and so $N+\pi^{-1}(D)=G$. By Lemma \ref{gdeltad1}(a) $\pi^{-1}(D)$ is $G_\alpha$-dense in $G$.
\end{proof}

The next proposition shows the stability under taking quotients of $d_\alpha$- and $c_\alpha$-extremality.

\begin{proposition}\label{quozd}%\emph{\cite{gal01,comfgal2}}
Let $\alpha$ be an infinite cardinal. Let $G$ be an $\alpha$-pseu\-do\-com\-pact abelian group and let $L$ be a closed subgroup of $G$. If $G$ is $d_\alpha$- (respectively, $c_\alpha$-) extremal, then $G/L$ is $d_\alpha$- (respectively, $c_\alpha$-) extremal.
\end{proposition}
\begin{proof}
Let $\pi:G\rightarrow G/L$ be the canonical projection. If $D$ is a $G_\alpha$-dense subgroup of $G/L$, by Lemma \ref{denspsc}(b) $\pi^{-1}(D)$ is $G_\alpha$-dense in $G$. Moreover $G/\pi^{-1}(D)$ is algebraically isomorphic to $(G/L)/D$.

Suppose that $G/L$ is not $d_\alpha$-extremal. Then there exists a $G_\alpha$-dense subgroup $D$ of $G/L$ such that $(G/L)/D$ is not divisible. Therefore $G/\pi^{-1}(D)$ is not divisible, hence $G$ is not $d_\alpha$-extremal.
If $G/L$ is not $c_\alpha$-extremal. Then there exists a $G_\alpha$-dense subgroup $D$ of $G/L$ such that $r_0((G/L)/D)\geq 2^\alpha$.
Consequently $r_0(G/\pi^{-1}(D))\geq 2^\alpha$, so $G$ is not $c_\alpha$-extremal.
\end{proof}

In this proposition we consider only $d_\alpha$- and $c_\alpha$-extremality. Indeed Theorem \ref{alpha-extremal-solution} and Example \ref{Example-cd} prove that, for $\alpha$-pseudocompact abelian groups, these are the only levels of extremality that are not equivalent to having weight $\leq\alpha$.

\subsection{The $P_\alpha$-topology}

A topological space $X$ is of \emph{first category} if it can be written as the union of countably many nowhere dense subsets of $X$. Moreover $X$ is of \emph{second category (Baire)} if it is not of first category, i.e., if for every family $\{U_n\}_{n\in\N}$ of open dense subsets of $X$, also $\bigcap_{n\in\N}U_n$ is dense in $X$.

\bigskip
Let $\alpha$ be an infinite cardinal and let $\tau$ be a topology on a set $X$. Then $P_\alpha\tau$ denotes the topology on $X$ generated by the $G_\alpha$-sets of $X$, which is called \emph{$P_\alpha$-topology}. Obviously $\tau\leq P_\alpha\tau$. If $X$ is a topological space, we denote by $P_\alpha X$ the set $X$ endowed with the $P_\alpha$-topology.

\medskip
The following theorem is the generalization for topological groups of \cite[Lemma 2.4]{CRob1} to the $P_\alpha$-topology.

\begin{theorem}\label{baire1}
Let $\alpha$ be an infinite cardinal and let $G$ be an $\alpha$-pseu\-do\-com\-pact group. Then $P_\alpha G$ is Baire.
\end{theorem}
\begin{proof}
Let $K$ be the completion of $G$. Then $K$ is compact by Theorem \ref{cr-theorem}.
Let $\mathcal B=\{x N: x\in K, N\in\Lambda_\alpha(K)\}$. By Corollary \ref{lambda2} $\mathcal B$ is a base of $P_\alpha K$. Consider a family $\{U_n\}_{n\in\N}$ of open dense subsets of $P_\alpha K$. We can chose them so that $U_n\supseteq U_{n+1}$ for every $n\in\N$. Let $A\in P_\alpha\tau$, $A\neq\emptyset$. Then $A\cap U_0$ is a non-empty element of $P_\alpha\tau$. Therefore there exists $B_0\in\mathcal{B}$ such that $B_0\subseteq A\cap U_0$. We proceed by induction. If $B_n\in\mathcal B$ has been defined, then $B_n\cap U_{n+1}$ is a non-empty open set in $P_\alpha K$ and so there exists $B_{n+1}\in\mathcal B$ such that $B_{n+1}\subseteq B_n\cap V_{n+1}$. Then $$A\cap\bigcap_{n\in\N}U_n=\bigcap_{n\in\N}(A\cap U_n)\supseteq\bigcap_{n\in\N}B_n.$$
Moreover $\bigcap_{n\in\N}B_n\neq\emptyset$, because $\{B_n\}_{n\in\N}$ is a decreasing sequence of closed subsets of $K$, which is compact.

By \cite[Lemma 2.4(b)]{CRob1} a $G_\delta$-dense subspace of a Baire space is Baire. Being $\alpha$-pseudocompact, $G$ is $G_\alpha$-dense in $K$ by Theorem \ref{cr-alpha}. Then $G$ is $G_\alpha$-dense in $P_\alpha K$, which is Baire by the previous part of the proof, and so $P_\alpha G$ is Baire.
\end{proof}

Let $G$ be an abelian group. In what follows $G^\#$ denotes $G$ endowed with the \emph{Bohr topology}, i.e., the initial topology of all elements of $\HG$.

\smallskip
The following lemma is the generalization to the $P_\alpha$-topology of \cite[Theorem 5.16]{CRob}.

\begin{lemma}\label{baire4}%\emph{\cite[Teorema 5.16]{comfrobiii}}%5.16
Let $\alpha$ be an infinite cardinal and let $(G,\tau)$ be a precompact abelian group such that every $h\in\HGemph$ is $P_\alpha\tau$-continuous. Then $(G,P_\alpha\tau)=P_\alpha G^\#$.
\end{lemma}
\begin{proof}
Let $\tau_G^\#$ be the Bohr topology on $G$, that is $G^\#=(G,\tau_G^\#)$. By the hypothesis $\tau_G^\#\leq P_\alpha\tau$. Then $P_\alpha \tau_G^\#\leq P_\alpha P_\alpha\tau=P_\alpha\tau$. Moreover $\tau\leq \tau_G^\#$ yields $P_\alpha\tau\leq P_\alpha \tau_G^\#$.
\end{proof}

In the first part of the theorem we give a more detailed description of the topology $P_\alpha G^\#$, while the equivalence of (a), (b) and (c) is the counterpart of \cite[Theorem 5.17]{CRob} for the $P_\alpha$-topology. Some ideas in the proof of this second part are similar to those in the proof of \cite[Theorem 5.17]{CRob}, but ours is shorter and simpler, thanks to the description of the topology $P_\alpha G^\#$ in algebraic terms. 

\begin{theorem}\label{baire5}%\emph{\cite[Teorema 5.17]{comfrobiii}}%5.17
Let $\alpha$ be an infinite cardinal and let $G$ be an abelian group. Then $\Lambda'_\alpha(G^\#):=\{N\leq G: |G/N|\leq 2^\alpha\}\subseteq\Lambda_\alpha(G^\#)$ is a local base at $0$ of $P_\alpha G^\#$. Consequently the following conditions are equivalent:
\begin{itemize}
	\item[(a)]$|G|\leq 2^\alpha$;
	\item[(b)]$P_\alpha G^\#$ is discrete;
	\item[(c)]$P_\alpha G^\#$ is Baire.
\end{itemize}
\end{theorem}
\begin{proof}
We prove first that $\Lambda''_\alpha(G^\#):=\{\bigcap_{\lambda<\alpha}\ker\chi_\lambda:\chi_\lambda\in\HG\}\subseteq \Lambda_\alpha(G^\#)$ is a local base at $0$ of $P_\alpha G^\#$ and then the equality $\Lambda'_\alpha(G^\#)=\Lambda_\alpha''(G^\#)$.

If $W$ is a $G_\alpha$-set of $G$ such that $0\in W$, then $W\supseteq \bigcap_{\lambda<\alpha}U_\lambda$, where each $U_\lambda$ is a neighborhood of $0$ in $G^\#$ belonging to the base. This means that $U_\lambda=\chi_\lambda^{-1}(V_\lambda)$, where $\chi_\lambda\in\HG$ and $V_\lambda$ is a neighborhood of $0$ in $\T$. Therefore $W\supseteq\bigcap_{\lambda<\alpha}\chi_\lambda^{-1}(0)=\bigcap_{\lambda<\alpha}\ker\chi_\lambda$. Note that each $\chi_\lambda^{-1}(0)$ is a $G_\delta$-set of $G^\#$, since $\{0\}$ is a $G_\delta$-set of $\T$ and hence $\bigcap_{\lambda<\alpha}\chi_\lambda^{-1}(0)$ is a $G_\alpha$-set of $G^\#$. Until now we have proved that $\Lambda_\alpha''(G^\#)$ is a local base at $0$ of $P_\alpha G^\#$. Moreover it is contained in $\Lambda_\alpha(G^\#)$, because each $ \bigcap_{\lambda<\alpha}\ker\chi_\lambda$, where $\chi_\lambda\in\HG$, is a closed $G_\alpha$-subgroup of $G^\#$.

It remains to prove that $\Lambda'_\alpha(G^\#)= \Lambda_\alpha''(G^\#)$. Let $N=\bigcap_{\lambda<\alpha}\ker\chi_\lambda$, where every $\chi_\lambda\in\HG$. Since for every $\lambda <\alpha$ there exists an injective homomorphism $G/\ker\chi_\lambda\to \T$, it follows that there exists an injective homomorphism $G/N\to\T^\alpha$. Then $|G/N|\leq 2^\alpha$. To prove the converse inclusion let $N\in\Lambda_\alpha'(G^\#)$. Then $N$ is closed in $G^\#$, because every subgroup of $G$ is closed in $G^\#$. Moreover, since $r(G/N)\leq 2^\alpha$, there exists an injective homomorphism $i:G/N\to \T^\alpha$; in fact, $G/N$ has an essential subgroup $B(G/N)$ (i.e., $B(G/N)$ non-trivially intersects each non-trivial subgroup of $G/N$) algebraically isomorphic to $\Z^{(r_0(G/N))}\oplus\bigoplus_{p\in\Prm}\Z(p)^{(r_p(G/N))}$. Since $r_0(\T^\alpha)=2^\alpha$ and $r_p(\T^\alpha)=2^\alpha$ for every $p\in \Prm$, there exists an injective homomorphism $B(G/N)\to\T^\alpha$ and by the divisibility of $\T^\alpha$ this homomorphism can be extended to $G/N$ (the extended homomorphism is still injective by the essentiality of $B(G/N)$ in $G/N$). Let $\pi:G\to G/N$ be the canonical projection and let $\pi_\lambda:\T^\alpha\to\T$ be the canonical projection for every $\lambda<\alpha$. Then $\chi_\lambda=\pi_\lambda\restriction_{i(G)}\circ i\circ\pi:G\to \T$ is a homomorphism. Moreover $N=\bigcap_{\lambda<\alpha}\ker\chi_\lambda$.

\medskip
It is clear that (a)$\Rightarrow$(b) by the first part of the proof and that (b)$\Rightarrow$(c).

(c)$\Rightarrow$(a) Suppose for a contradiction that $|G|> 2^\alpha$. We prove that $P_\alpha G^\#$ is of first category.
Let $D(G)=\mathbb{Q}^{(r_0(G))}\oplus\bigoplus_{p\in\Prm}\Z(p^{\infty})^{(r_p(G))}$ be the divisible hull of $G$. Moreover $G$ has a subgroup algebraically isomorphic to $B(G)= \mathbb{Z}^{(r_0(G))}\oplus\bigoplus_{p\in\Prm}\Z(p)^{(r_p(G))}$. We can think $B(G)\leq G\leq D(G)\leq\T^I$, where $|I|=r(G)=r_0(G)+\sup_{p\in\Prm}r_p(G)$, because $\Q$ and $\Z(p^\infty)$ are algebraically isomorphic to subgroups of $\T$ for every $p\in\Prm$. 
Since $|G|>2^\alpha$, it follows that $|I|>2^\alpha$.

For $x\in G$ let $s(x)=\{i\in I:x_i\neq 0\}$ and for $n\in\N$ we set $$A(n)=\{x\in G:|s(x)|\leq n\}.$$  Then $G=\bigcup_{n\in\N}A(n)$.

For every $n\in\N$ we have $A(n)\subseteq A(n+1)$. We prove that $A(n)$ is closed in the topology $\tau$ induced on $G$ by $\T^I$ for every $n\in\N$: it is obvious that $A(0)$ is compact and $A(1)$ is compact, because every open neighborhood of $0$ in $(G,\tau)$ contains all but a finite number of elements of $A(1)$. Moreover, for every $n\in\N$ with $n>1$, $A(n)$ is the sum of $n$ copies of $A(1)$ and so it is compact. 

To conclude the proof we have to prove that each $A(n)$ has empty interior in $P_\alpha G^\#$. Since $A(n)\subseteq A(n+1)$ for every $n\in\N$, it suffices to prove that $\text{Int}_{P_\alpha G^\#}(A(n))$ is empty for sufficiently large $n\in\N$; we consider $n\in\N_+$.
By the first part of the proof, it suffices to show that if $x\in G$ and $N$ is a subgroup of $G$ such that $|G/N|\leq 2^\alpha$, then $x\in x+N\not\subseteq A(n)$ for all $n\in\N_+$. Moreover we can suppose that $x=0$. In fact, if there exist $x\in G$ and $N\leq G$ with $|G/N|\leq 2^\alpha$, such that $x+N \subseteq A(n)$, then $x=x+0\in A(n)$ and $0\in N\subseteq -x+A(n)\subseteq A(n)+A(n)\subseteq A(2n).$
So let $N\leq G$ be such that $|G/N|\leq 2^\alpha$. Let $\{I_{\xi}:\xi<(2^\alpha)^+\}$ be a family of subsets of $I$ such that $|I_{\xi}|=n\ \text{and}\ I_{\xi}\cap I_{\xi'}=\emptyset$ for every $\xi<\xi'<(2^\alpha)^+$.
For every $\xi< (2^\alpha)^+$ there exists $x_\xi \in B(G)\leq G$ such that $s(x_\xi)=I_{\xi}$. Let $\pi:G\to G/N$ be the canonical projection. Since $|\{x_\xi:\xi<(2^\alpha)^+\}|=(2^\alpha)^+>2^\alpha\geq|G/N|$, it follows that there exist $\xi<\xi'<(2^\alpha)^+$ such that $\pi(x_\xi)=\pi(x_{\xi'})$. Then $x_\xi-x_{\xi'}\in\ker \pi=N$. But $s(x_\xi-x_{\xi'})=I_{\xi}\cup I_{\xi'}$ and $|I_{\xi}\cup I_{\xi'}|=2n$. Hence $x_\xi -x_{\xi'}\not\in A(n)$.
\end{proof}

\section{Construction of $G_\alpha$-dense subgroups}

The following lemma is a generalization to $\alpha$-pseudocompact abelian groups of \cite[Lemma 2.13]{CvM}. The construction is the same.

\begin{lemma}\label{cconn0}
Let $\alpha$ be an infinite cardinal. Let $G$ be an $\alpha$-pseudocompact abelian group and $G=\bigcup_{n\in\N}A_n$, where all $A_n$ are subgroups of $G$. Then there exist $n\in\N$ and $N\in\Lambda_\alpha(G)$ such that $A_n\cap N$ is $G_\alpha$-dense in $N$.
\end{lemma}
\begin{proof}
Since $(G,P_\alpha\tau)$ is Baire by Theorem \ref{baire1} and since $G=\bigcup_{n\in\N}\overline{A_n}^{P_\alpha\tau}$, there exists $n\in\N$ such that $\text{Int}_{P_\alpha\tau}\overline{A_n}^{P_\alpha\tau}\neq \emptyset$. The family $\{x+N:x\in G,\ N\in\Lambda_\alpha(G)\}$ is a base of $P_\alpha\tau$ by Corollary \ref{lambda2}; consequently there exist $x\in G$ and $N\in \Lambda_\alpha(G)$ such that $x+N\subseteq \overline{A_n}^{P_\alpha\tau}$.
Since $x+N$ is open and closed in $P_\alpha\tau$, then $\overline{A_n\cap(x+N)}^{P_\alpha\tau}=x+N$, i.e., $A_n\cap(x+N)$ is $G_\alpha$-dense in $x+N$.

We can suppose without loss of generality that $x\in A_n$, because we can choose $a\in A_n$ such that $a+N=x+N$. In fact, since $A_n\cap (x+N)\neq \emptyset$, because $A_n\cap (x+N)$ is $G_\alpha$-dense in $x+N$, it follows that there exists $a\in (x+N)\cap A_n$. In particular $a\in x+N$ and so $a+N=x+N$.

We can choose $x=0$ because all $A_n\leq G$: since $A_n\cap(x+N)$ is $G_\alpha$-dense in $x+N$, it follows that $(A_n-x)\cap N$ is $G_\alpha$-dense in $N$ and $A_n-x=A_n$ since $x\in A_n$.
\end{proof}

For $\alpha=\omega$ the next lemma is \cite[Lemma 4.1(b)]{CGvM}.

\begin{lemma}\label{lambda}
Let $\alpha$ be an infinite cardinal and let $G$ be an $\alpha$-pseu\-do\-com\-pact abelian group. If $N\in \Lambda_\alpha (G)$ and $D$ is $G_\alpha$-dense in $N$, then there exists a subgroup $E$ of $G$ such that $|E|\leq 2^\alpha$ and $D+E$ is $G_\alpha$-dense in $G$.
\end{lemma}
\begin{proof}
Since $D$ is $G_\alpha$-dense in $N\in\Lambda_\alpha(G)$, it follows that $x+D$ is $G_\alpha$-dense in $x+N$ for every $x\in G$. By Theorem \ref{lambda4} $G/N$ is compact of weight $\alpha$ and so $|G/N|\leq 2^\alpha$, i.e., there exists $X\subseteq G$ with $|X|\leq 2^\alpha$ such that $G/N=\{x+N:x\in X\}$. We set $E=\langle X \rangle$; then $|E|\leq 2^\alpha$ and $D+E$ is $G_\alpha$-dense in $G$.
\end{proof}

Lemmas \ref{cconn0} and \ref{lambda} imply that in case $G$ is an $\alpha$-pseudocompact abelian group such that $G=\bigcup_{n\in\N}A_n$, where all $A_n$ are subgroups of $G$, then there exist $n\in\N$, $N\in\Lambda_\alpha(G)$ and $E\leq G$ with $|E|\leq 2^\alpha$ such that $(A_n\cap N)+E$ is $G_\alpha$-dense in $G$. In particular we have the following useful result.

\begin{corollary}\label{lambda+}
Let $\alpha$ be an infinite cardinal. Let $G$ be an $\alpha$-pseudocompact abelian group such that $G=\bigcup_{n\in\N}A_n$, where all $A_n\leq G$. Then there exist $n\in\N$ and a subgroup $E$ of $G$ such that $|E|\leq 2^\alpha$ and $A_n+E$ is $G_\alpha$-dense in $G$.
\end{corollary}

Thanks to the previous results we prove the following theorem, which gives a first restriction for extremal $\alpha$-pseudocompact groups, i.e., the free rank cannot be too big. The case $\alpha=\omega$ of this theorem is \cite[Theorem 3.6]{DGM}. That theorem in its own terms was inspired by \cite[Theorem 5.10 (b)]{CGal2} and used ideas from the proof of \cite[Proposition 4.4]{CGvM}.

\begin{theorem}\label{G>c2}
Let $\alpha$ be an infinite cardinal and let $G$ be an $\alpha$-pseu\-do\-com\-pact abelian group. If $G$ is $c_\alpha$-extremal, then $r_0(G)\leq 2^\alpha$.
\end{theorem}
\begin{proof}
Let $\kappa=r_0(G)$ and let $M$ be a maximal independent subset of $G$ consisting of non-torsion elements. Then $|M|=\kappa$ and there exists a partition $M=\bigcup_{n\in\N_+} M_n$ such that $|M_n|=\kappa$ for each $n\in\N_+$. Let $U_n=\langle M_n\rangle$, $V_n=U_1\oplus\ldots\oplus U_n$ and $A_n=\{x\in G:n!x\in V_n\}$ for every $n\in \N_+$. Then $G=\bigcup_{n\in\N_+}A_n$. By Corollary \ref{lambda+} there exist $n\in \N_+$ and a subgroup $E$ of $G$ such that $D=A_n+E$ is $G_\alpha$-dense in $G$ and $|E|\leq 2^\alpha$. Hence $|E/(A_n\cap E)|\leq 2^\alpha$. Since $D/A_n=(A_n+E)/A_n$ is algebraically isomorphic to $E/(A_n\cap E),$ it follows that $|D/A_n|\leq 2^\alpha$. Since $G$ is $c_\alpha$-extremal, it follows that $r_0(G/D)<2^\alpha$ and so $r_0(G/A_n)\leq 2^\alpha$ because $(G/A_n)/(D/A_n)$ is algebraically isomorphic to $G/D$. On the other hand, $r_0(G/A_n)\geq\kappa$, as $U_{n}$ embeds into $G/A_n$. Hence $\kappa\leq 2^\alpha$.
\end{proof}

\section{The dense graph theorem}\label{dense-graph-sec}

For $\alpha=\omega$ the following lemma is \cite[Lemma 3.7]{DGM}. The idea of this lemma comes from the proof of \cite[Theorem 4.1]{CRob}.

\begin{lemma}\label{str0'}
Let $\alpha$ be an infinite cardinal. Let $G$ be a topological abelian group and $H$ a compact abelian group with $|H|>1$ and $w(H)\leq\alpha$. Let $h: G \rightarrow H$ be a surjective homomorphism. Then $\Gamma_h$ is $G_\alpha$-dense in $G\times H$ if and only if $\ker h$ is $G_\alpha$-dense in $G$. 
\end{lemma}
\begin{proof}
Suppose that $\Gamma_h$ is $G_\alpha$-dense in $G\times H$. Let $W$ be a non-empty $G_\alpha$-set of $G$. Since $W\times\{0\}$ is a $G_\alpha$-set of $G\times H$, then $\Gamma_h\cap(W\times\{0\})\neq\emptyset$. But $\Gamma_h\cap(W\times\{0\})=(W\cap\ker h)\times\{0\}$ and so $W\cap\ker h\neq\emptyset$. This proves that $\ker h$ is $G_\alpha$-dense in $G$. %If $\ker h$ is not proper in $G$, i.e., $\ker h=G$, then $h\equiv 0$, which is not possible.

Suppose that $\ker h$ is $G_\alpha$-dense in $G$. Every non-empty $G_\alpha$-set of $G\times H$ contains a $G_\alpha$-set of $G\times H$ of the form $W\times\{y\}$, where $W$ is a non-empty $G_\alpha$-set of $G$ and $y\in H$. Since $h$ is surjective, there exists $z\in G$ such that $h(z)=y$. Also $z+\ker h$ is $G_\alpha$-dense in $G$ and so $W\cap(z+\ker h)\neq\emptyset$. Consequently there exists $x\in W\cap(z+\ker h)$. From $x\in z+\ker h$ it follows that $x-z\in\ker h$. Therefore $h(x-z)=0$ and so $h(x)=h(z)=y$. Since $x\in W$, it follows that $(x,y)\in (W\times\{y\})\cap\Gamma_h$. This proves that $(W\times\{y\})\cap\Gamma_h\neq\emptyset$. Hence $\Gamma_h$ is $G_\alpha$-dense in $G\times H$.
\end{proof}

In this lemma we conclude that $\ker h$ is proper in $G$, from the hypotheses that $h$ is surjective and $H$ is not trivial.

\medskip
The next remark explains the role of the graph of a homomorphism in relation to the topology of the domain (see \cite[Remark 2.12]{DGM} for more details).

\begin{remark}\label{ossgraf3}
Let $(G,\tau)$ and $H$ be topological groups and $h:(G,\tau)\rightarrow H$ a homomorphism. Consider the map $j:G\rightarrow\Gamma_h$ such that $j(x)=(x,h(x))$ for every $x\in G$. Then $j$ is an open isomorphism. Endow $\Gamma_h$ with the group topology induced by the product $(G,\tau)\times H$. The topology $\tau_h$ is {\it the weakest group topology on $G$ such that $\tau_h\geq\tau$ and for which $j$ is continuous}. Then $j:(G,\tau_h)\rightarrow\Gamma_h$ is a homeomorphism. Moreover $\tau_h$ is the weakest group topology on $G$ such that $\tau_h\geq\tau$ and for which $h$ is continuous. Clearly $h$ is $\tau$-continuous if and only if $\tau_h=\tau$.
\end{remark}

The following theorem gives a necessary condition for an $\alpha$-pseu\-do\-com\-pact group to be either $s_\alpha$- or $r_\alpha$-extremal. We call it ``dense graph theorem'' because of the nature of this necessary condition. Moreover it is the generalization to $\alpha$-pseudocompact groups of \cite[Theorem 4.1]{CRob}.

\begin{theorem}\label{dense-graph-alpha}
Let $\alpha$ be an infinite cardinal and let $(G,\tau)$ be an $\alpha$-pseu\-do\-com\-pact group such that there exists a homomorphism $h:G\to H$ where $H$ is an $\alpha$-pseudocompact abelian group with $|H|>1$ and $\Gamma_h$ is $G_\alpha$-dense in $(G,\tau)\times H$. Then:
\begin{itemize}
	\item[(a)]there exists an $\alpha$-pseudocompact group topology $\tau'>\tau$ on $G$ such that $w(G,\tau')=w(G,\tau)$;
	\item[(b)]there exists a proper $G_\alpha$-dense subgroup $D$ of $(G,\tau)$ such that $w(D)=w(G,\tau)$.
\end{itemize}
\end{theorem}
\begin{proof}
We first prove that $H$ can be chosen compact of weight $\leq\alpha$ and that in such a case $h$ is surjective.
There exists a continuous character $\chi:H\rightarrow\T$ such that $\chi (H)\neq\{0\}$. Let $H'=\chi(H)\subseteq\T$ and $h'=\chi\circ h$. Then $H'$ is compact and metrizable. So $H'$ is either $\T$ or $\Z(n)\leq\T$ for some integer $n>1$. Since $1_G\times \chi:(G,\tau)\times H\to (G,\tau)\times H'$ is a continuous surjective homomorphism such that $(1_G\times\chi)(\Gamma_h)=\Gamma_{h'}$ and $\Gamma_h$ is $G_\alpha$-dense in $(G,\tau)\times H$, it follows that $\Gamma_{h'}$ is $G_\alpha$-dense in $(G,\tau)\times H'$. Let $p_2:G\times H'\to H'$ be the canonical projection. Then $p_2(\Gamma_{h'})=h'(G)$ is $G_\alpha$-dense in $H'$, which is metrizable. Hence $h'(G)=H'$ and $h'$ is surjective.

(a) Since $G$ is $G_\alpha$-dense in $\widetilde {(G,\tau)}$ by Theorem \ref{cr-alpha} and since $\Gamma_h$ is $G_\alpha$-dense in $(G,\tau)\times H$, it follows that $\Gamma_h$ is $G_\alpha$-dense in $\widetilde{(G,\tau)}\times H$. Consequently $\Gamma_h$ with the topology inherited from $\widetilde{(G,\tau)}\times H$ is $\alpha$-pseudocompact in view of Corollary \ref{alpha-psc<->G_alpha-dense}. Using Remark \ref{ossgraf3} let $\tau_h$ be the coarsest group topology on $G$ such that $\tau_h\geq\tau$ and $h$ is $\tau_h$-continuous; then $(G,\tau_h)$ is topologically isomorphic to $\Gamma_h$ and so it is $\alpha$-pseudocompact. If $\tau_h=\tau$, then $h$ is continuous and the closed graph theorem yields that $\Gamma_h$ is closed in $(G,\tau)\times H$. This is not possible because $\Gamma_h$ is dense in $(G,\tau)\times H$ by the hypothesis. Hence $\tau_h\gneq\tau$. By the hypothesis $w(G,\tau)>\omega$ and since $H$ is metrizable, then $$w(G,\tau_h)=w(\Gamma_h)=w((G,\tau)\times H)=w(G,\tau)\cdot w(H)=w(G,\tau).$$

(b) Let $D=\ker h$. By Lemma \ref{str0'} $D$ is $G_\alpha$-dense in $(G,\tau)$. Moreover $D$ is proper in $G$. Clearly $w(D)=w(G,\tau)$.
\end{proof}

The next theorem shows that $\alpha$-extremality ``puts together" $s_\alpha$- and $r_\alpha$-extremality. It is the generalization to $\alpha$-pseudocompact abelian groups of \cite[Theorem 3.12]{DGM}.

\begin{theorem}\label{debestr}
Let $\alpha$ be an infinite cardinal and let $G$ be an $\alpha$-pseu\-do\-com\-pact abelian group which is either $s_\alpha$- or $r_\alpha$-extremal. Then $G$ is $\alpha$-extremal.
\end{theorem}
\begin{proof}
Suppose looking for a contradiction that $G$ is not $\alpha$-extremal. Then there exists a dense $\alpha$-pseudocompact subgroup $D$ of $G$ such that either $G/D$ is not divisible or $r_0(G/D)\geq 2^\alpha$. In both cases $D$ has to be a proper dense $\alpha$-pseudocompact subgroup of $G$. Then $G$ is not $s_\alpha$-extremal. We prove that $G$ is not $r_\alpha$-extremal as well. Note that $D$ is $G_\alpha$-dense in $G$ by Corollary \ref{alpha-psc<->G_alpha-dense}.
Let $\pi:G\rightarrow G/D$ be the canonical projection.

Now we build a surjective homomorphism $h:G/D \to H$, where $H$ is compact $|H|>1$ and $\ker h$ is $G_\alpha$-dense in $G$. By assumption $D$ is a $G_\alpha$-dense subgroup of $G$ such that either $G/D$ is not divisible or $r_0(G/D)\geq 2^\alpha$. In the first case $G/D$ admits a non-trivial finite quotient $H$, while in the second case we can find a surjective homomorphism $G/D \to \T=H$ as $|\T|\leq r_0(G/D)$. Since $\ker h$ contains $D$ in both cases, $\ker h$ is $G_\alpha$-dense in $G$. Apply Lemma \ref{str0'} and Theorem \ref{dense-graph-alpha} to conclude that $G$ is not $r_\alpha$-extremal.
\end{proof}

The following proposition and lemma are the generalizations to the $\alpha$-pseudocompact case of \cite[Theorems 5.8 and 5.9]{CRob} respectively. The ideas used in the proofs are similar. The next claim is needed in the proofs of both.

\begin{claim}\label{prime_index}
Let $p\in\Prm$, let $G$ be an abelian group of exponent $p$ and $h:G\to\Z(p)\leq\T$ a continuous surjective homomorphism. Then $\Gamma_h$ has index $p$ in $G\times \Z(p)$.
\end{claim}
\begin{proof}
Consider $\xi:G\times\Z(p)\rightarrow\Z(p)$, defined by $\xi(g,y)=h(g)-y$ for all $(g,y)\in G\times\Z(p)$. Then $\xi$ is surjective and $\ker\xi=\Gamma_h$. Therefore $G\times\Z(p)/\ker\xi=G\times\Z(p)/\Gamma_h$ is algebraically isomorphic to $\Z(p)$ and so they have the same cardinality $p$.
\end{proof}

The following proposition shows that for $\alpha$-pseudocompact abelian gro\-ups of prime exponent $s_\alpha$-extremality is equivalent to $r_\alpha$-extremality.

\begin{proposition}\label{pgr1}%\emph{\cite[Teorema 5.8]{comfrobiii}}%5.8
Let $\alpha$ be an infinite cardinal and let $(G,\tau)$ be an $\alpha$-pseu\-do\-com\-pact abelian group of exponent $p\in\Prm$. Then the following conditions are equivalent:
\begin{itemize}
\item[(a)]there exist an $\alpha$-pseudocompact abelian group $H$ with $|H|>1$ and a homomorphism $h:G\to H$ such that $\Gamma_h$ is $G_\alpha$-dense in $(G,\tau)\times H$;
\item[(b)]$(G,\tau)$ is not $s_\alpha$-extremal;
\item[(c)]$(G,\tau)$ is not $r_\alpha$-extremal.
\end{itemize}
\end{proposition}
\begin{proof}
(a)$\Rightarrow$(b) and (a)$\Rightarrow$(c) follows from Theorem \ref{dense-graph-alpha}.

(b)$\Rightarrow$(c) Suppose that $(G,\tau)$ is not $s_\alpha$-extremal. Then there exists a proper dense $\alpha$-pseudocompact subgroup $D$ of $(G,\tau)$. We can suppose without loss of generality that $D$ is maximal and so that $|G/D|=p$. Let $\tau'$ be the coarsest group topology such that $\tau'\supseteq\tau\cup\{x+D:x\in G\}$.
Since $D\not\in\tau$ but $D\in\tau'$, so $\tau'>\tau$. Since $(D,\tau'\restriction_D)=(D,\tau\restriction_D)$ and $D$ is an $\alpha$-pseudocompact subgroup of $(G,\tau)$, it follows that $D$ is an $\alpha$-pseudocompact subgroup of $(G,\tau')$. Hence $(G,\tau')$ is $\alpha$-pseudocompact by Lemma \ref{finite_index}.

(c)$\Rightarrow$(a) Suppose that $G$ is not $r_\alpha$-extremal. Then there exists an $\alpha$-pseudocompact group topology $\tau'$ on $G$ such that $\tau'>\tau$. Since both topologies are precompact, there exists an homomorphism $h:G\to\T$ such that $h$ is $\tau'$-continuous but not $\tau$-continuous. Note that $h(G)\neq\{0\}$. Being $G$ of exponent $p$, so $h(G)=\Z(p)\leq\T$. Let $$H=\Z(p).$$ Since $h$ is not $\tau$-continuous, by the closed graph theorem $\Gamma_h$ is not closed in $(G,\tau)\times\Z(p)$. Moreover $|(G,\tau)\times\Z(p)/\Gamma_h|=p$ by Claim \ref{prime_index}. Since $\Gamma_h$ is a subgroup of index $p$ in $(G,\tau)\times\Z(p)$ and it is not closed, then $\Gamma_h$ is dense in $(G,\tau)\times\Z(p)$. %by Claim \ref{prime_index->dense_or_closed}.

Endow $G$ with the topology $\tau_h$, that is the coarsest group topology on $G$ such that $\tau_h\geq\tau$ and $h$ is $\tau_h$-continuous (see Remark \ref{ossgraf3}). Then $(G,\tau_h)$ is $\alpha$-pseudocompact, because $h$ is $\tau'$-continuous and so $\tau_h\leq\tau'$ and $\tau'$ is $\alpha$-pseudocompact. By Remark \ref{ossgraf3} $\Gamma_h$ is topologically isomorphic to $(G,\tau_h)$ and so $\Gamma_h$ is $\alpha$-pseudocompact. Since $\Gamma_h$ is dense and $\alpha$-pseudocompact in $(G,\tau)\times \Z(p)$, Corollary \ref{alpha-psc<->G_alpha-dense} yields that $\Gamma_h$ is $G_\alpha$-dense in $(G,\tau)\times\Z(p)=(G,\tau)\times H$. 
\end{proof}

\begin{lemma}\label{pgr3}%\emph{\cite[Teorema 5.9]{comfrobiii}}%5.9
Let $\alpha$ be an infinite cardinal. Let $(G,\tau)$ be an $\alpha$-pseu\-do\-com\-pact abelian group of exponent $p\in\Prm$ such that $G$ is either $s_\alpha$- or $r_\alpha$-extremal. Then every $h\in\HGemph$ is $P_\alpha\tau$-continuous (i.e., $\HGemph\subseteq\widehat{(G,P_\alpha\tau)}$).
\end{lemma}
\begin{proof}
If $h\equiv 0$, then $h$ is $P_\alpha\tau$-continuous. Suppose that $h\not\equiv 0$. Then $h(G)=\Z(p)\leq\T$. Since $G$ is either $s_\alpha$- or $r_\alpha$-extremal, by Proposition \ref{pgr1} $\Gamma_h$ is not $G_\alpha$-dense in $(G,\tau)\times\Z(p)$, that is $\Gamma_h$ is not dense in $P_\alpha((G,\tau)\times\Z(p))=P_\alpha(G,\tau)\times P_\alpha\Z(p)=(G,P_\alpha\tau)\times\Z(p)$. Claim \ref{prime_index} implies that $|G\times\Z(p)/\Gamma_h|=p$. Since $\Gamma_h$ is not dense and is of prime index in $(G,P_\alpha\tau)\times \Z(p)$, it follows that $\Gamma_h$ is closed in $(G,P_\alpha\tau)\times\Z(p)$. By the closed graph theorem $h$ is $P_\alpha\tau$-continuous.
\end{proof}

\section{The $\alpha$-singular groups}\label{alpha-singular-sec}

\subsection{The torsion abelian groups and extremality}\label{torsion-sec}

The following definition and two lemmas are the generalization to the $\alpha$-pseu\-do\-com\-pact case of \cite[Notation 5.10, Theorem 5.11 and Lemma 5.13]{CRob}. The constructions are almost the same.

\begin{deff}
Let $\alpha$ be an infinite cardinal, let $X$ be a topological space and $Y\subseteq X$. The \emph{$\alpha$-closure} of $Y$ in $X$ is $\alpha\text{-cl}_X(Y)=\bigcup\{\overline{Y}^X:A\subseteq Y,\ |A|\leq\alpha\}.$
\end{deff}

For $Y\subseteq X$, the set $\alpha\text{-cl}_X(Y)$ is \emph{$\alpha$-closed} in $X$, i.e., $\alpha\text{-cl}_X(\alpha\text{-cl}_X(Y))=\alpha\text{-cl}_X(Y)$.

\begin{lemma}\label{omegacl}%\emph{\cite[Teorema 5.11]{comfrobiii}}%5.11
Let $\alpha$ be an infinite cardinal. Let $G$ be an $\alpha$-pseu\-do\-com\-pact abelian group and let $h\in\HGemph$. Then the following conditions are equivalent:
\begin{itemize}
\item[(a)]$h\in\alpha\text{-cl}_{\HGemph}\widehat{G}$;
\item[(b)]there exists $N\in\Lambda_\alpha(G)$ such that $N\subseteq\ker h$.
\end{itemize}
\end{lemma}
\begin{proof}
(a)$\Rightarrow$(b) Suppose that $h\in\alpha\text{-cl}_{\HG}\widehat{G}$. Let $A\subseteq\widehat{G}$ such that $|A|\leq\alpha$ and $h\in \overline A^{\HG}$. We set $N=\bigcap\{\ker f:f\in A\}$. Then $N\in\Lambda_\alpha(G)$. Moreover $N\subseteq\ker h$ as $h\in\overline A^{\HG}$.

(b)$\Rightarrow$(a) Let $N\in\Lambda_\alpha(G)$ and let $\pi:G\to G/N$ be the canonical projection. The group $G/N$ is compact of weight $\leq\alpha$ and so $|\widehat{G/N}|=w(G/N)\leq\alpha$ by Fact \ref{dens1}(c). We enumerate the elements of $\widehat{G/N}$ as $\widehat{G/N}=\{\chi_\lambda:\lambda<\alpha\}$ and define $A=\{\chi_\lambda\circ\pi:\lambda<\alpha\}\leq\widehat{G}.$ We prove that $h\in \overline A^{\HG}$. Suppose that $h\not\in\overline A^{\HG}$. Since $A$ is a closed subgroup of the compact group $\HG$, there exists $\xi\in\widehat{\HG}$ such that $\xi(h)\neq 0$ and $\xi(f)=0$ for every $f\in A$. By the Pontryagin duality there exists $x\in G$ such that $f(x)=\xi(f)$ for every $f\in\widehat{G}$. Then $\chi_\lambda(\pi(x))=\xi(\chi_\lambda\circ \pi)=0$ for every $\chi_\lambda\in\widehat{G/N}$. Then $\pi(x)=x+N=N$ and so $x\in N\subseteq\ker h$, i.e., $h(x)=0$. But $h(x)=\xi(h)\neq 0$, a contradiction.
\end{proof}

\begin{lemma}\label{pgr4}%\emph{\cite[Lemma 5.13]{comfrobiii}}%5.13
Let $\alpha$ be an infinite cardinal. Let $(G,\tau)$ be an $\alpha$-pseu\-do\-com\-pact abelian group of exponent $p\in\Prm$ such that $G$ is either $s_\alpha$- or $r_\alpha$-extremal and $|G|=\beta\geq\alpha$. Then for the completion $K$ of $(G,\tau)$:
\begin{itemize}
\item[(a)]for every $h\in\HGemph$ there exists $h'\in\text{Hom}(K,\T)$ such that $h'\restriction_G=h$ and $h'$ is $P_\alpha K$-continuous;
\item[(b)]$\HGemph\subseteq\alpha\text{-cl}_{\HGemph}\widehat{(G,\tau)}$;
\item[(c)]$\psi(G,\tau)\leq\alpha\cdot\log\beta$.
\end{itemize}
\end{lemma}
\begin{proof}
(a) Since $G$ is $G_\alpha$-dense in $K$ by Theorem \ref{cr-alpha}, it follows that $G$ is dense in $P_\alpha K$. By Lemma \ref{pgr3} every $h\in \HG$ is $P_\alpha\tau$-continuous. Therefore $h$ can be extended to $h'\in\text{Hom}(K,\T)$, such that $h'$ is $P_\alpha K$-continuous, because $(G,P_\alpha\tau)$ is dense in $P_\alpha K$.

(b) Since $h\in\HG$, by (a) there exists $h'\in\text{Hom}(K,\T)$ such that $h'\restriction_G=h$ and $h'$ is $P_\alpha K$-continuous. By Lemma \ref{omegacl} $h'\in\alpha\text{-cl}_{\text{Hom}(K,\T)}\widehat{K}.$ Therefore $h\in\alpha\text{-cl}_{\HG}\widehat{(G,\tau)}$. Indeed, let $A'\subseteq\widehat{K}$ be such that $|A'|\leq\alpha$ and $h'\in\overline {A'}^{\text{Hom}(K,\T)}$. For $f'\in A'$ we set $f=f'\restriction_G\in\widehat{(G,\tau)}$ and $A=\{f'\restriction_G:f'\in A'\}$. There exists a net $\{f'_{\lambda}\}_{\lambda}$ in $A'$ such that $f_\lambda'\to h'$ in $\text{Hom}(K,\T)$; since the topology on $\text{Hom}(K,\T)$ is the point-wise convergence topology, this means that $f'_{\lambda}(x)\rightarrow h'(x)$ for every $x\in K$. Then $f_{\lambda}(x)\rightarrow h(x)$ for every $x\in G$. Hence $f_\lambda\to h$ in $\HG$ and so $h\in\overline A^{\HG}\subseteq\alpha\text{-cl}_{\HG}\widehat{(G,\tau)}$.

(c) By Fact \ref{dens1}(a),(c) $d(\HG)=\log w(\HG)\leq\log\beta$. Then there exists a dense subset $S$ of $\HG$ such that $|S|\leq\log\beta$. By (b) for every $h\in S$ there exists $A(h)\subseteq\widehat{(G,\tau)}$ such that $|A(h)|\leq\alpha$ and $h\in \overline{A(h)}^{\HG}$. Then $A:=\bigcup\{A(h):h\in S\}$ is dense in $\HG$, because $S\subseteq\overline A^{\HG}$. Moreover $A\subseteq \widehat{(G,\tau)}$ and $|A|\leq\alpha\cdot\log\beta$. 
Let $x\in G\setminus\{0\}$ and let $\{V_n:n\in\N\}$ be a local base at $0$ of $\T$. Since $A$ is dense in $\HG$, it separates the points of $G$ and so there exists $f\in A$ such that $f(x)\neq 0$. Then there exists $n\in\N$ such that $f(x)\not\in V_n$. Therefore $\bigcap_{n\in\N,f\in A}f^{-1}(V_n)=\{0\}$ and hence $\psi(G,\tau)\leq|A|\leq\alpha\cdot\log\beta$.
\end{proof}

Now we prove Theorem \ref{alpha-extremal-solution} in the torsion case. For $\alpha=\omega$ it implies \cite[Corollary 7.5]{CRob} and the proof is inspired by that of the result from \cite{CRob}. Since every torsion $\alpha$-pseudocompact group is $c_\alpha$-extremal, we observe that a torsion $\alpha$-pseudocompact abelian group is $\alpha$-extremal if and only if it is $d_\alpha$-extremal.

\begin{theorem}\label{torsion}%\emph{\cite[Corollario 7.5]{comfrobiii}}%7.5
Let $\alpha$ be an infinite cardinal and let $G$ be an $\alpha$-pseu\-do\-com\-pact torsion abelian group. Then $G$ is $\alpha$-extremal if and only if $w(G)\leq\alpha$.
\end{theorem}
\begin{proof}
If $w(G)\leq \alpha$, then $G$ is $\alpha$-extremal by Proposition \ref{ccomp} and Theorem \ref{debestr}.

Suppose that $w(G)>\alpha$.
We prove that there exists $p\in\Prm$ such that $w(G/\overline{p G})>\alpha$. Since $G$ is torsion, then it is bounded-torsion by \cite[Lemma 7.4]{CRob}. Therefore $K=\widetilde G$ is bounded-torsion. Consequently $K$ is topologically isomorphic to $\prod_{p\in\Prm}t_p(K)$, where $t_p(K)=\{x\in K:p^n x=0\ \text{for some}\ n\in\N_+\}$. Since $w(K)=\max_{p\in\Prm}w(t_p(K))$ and $w(K)=w(G)>\alpha$, there exists $p\in\Prm$ such that $w(t_p(K))>\alpha$. Moreover for this $p$ we have $w(t_p(K))=w(K_{(p)})$, where $K_{(p)}=t_p(K)/p t_p(K)$, by \cite[Lemma 4.1(b)]{DG}. Consider the composition $\varphi_p$ of the canonical projections $K\to t_p(K)$ and $t_p(K)\to K_{(p)}$. Since $G$ is dense in $K$, it follows that $\varphi_p(G)$ is dense in $K_{(p)}$ and so $w(\varphi_p(G))=w(K_{(p)})>\alpha$. Moreover there exists a continuous isomorphism $G/(\ker\varphi_p\cap G)\to\varphi_p(G)$. Since $\ker\varphi_p=p K$ and $p K\cap G=\overline{p G}$, there exists a continuous isomorphism $G/\overline{p G}\to \varphi_p(G)$. Hence $w(G/\overline{p G})\geq w(\varphi_p(G))>\alpha$.

Let $G_1=G/\overline{p G}$. We prove that $G_1$ is not $s_\alpha$-extremal. Suppose for a contradiction that $G_1$ is $s_\alpha$-extremal. By Lemma \ref{pgr3} every $h\in\text{Hom}(G_1,\T)$ is $P_\alpha\tau$-continuous, and so $P_\alpha G_1=P_\alpha G_1^\#$ by Lemma \ref{baire4}. By Theorem \ref{baire1} $P_\alpha G_1$ is Baire, hence $|G_1|\leq 2^\alpha$ by Theorem \ref{baire5}. By Lemma \ref{pgr4}(c) $\psi(G_1)\leq\alpha\cdot\log 2^\alpha=\alpha$ and so Lemma \ref{psi->w} implies $w(G_1)=\psi(G_1)\leq\alpha$; this contradicts our assumption.

Then there exists a proper dense $\alpha$-pseudocompact subgroup $D$ of $G_1$. By Corollary \ref{alpha-psc<->G_alpha-dense} $D$ is  $G_\alpha$-dense in $G_1$. Let $\pi:G\to G_1$ be the canonical projection. By Lemma \ref{denspsc}(b) $\pi^{-1}(D)$ is a proper $G_\alpha$-dense subgroup of $G$, then dense $\alpha$-pseudocompact in $G$ by Corollary \ref{alpha-psc<->G_alpha-dense}. Since $G/\pi^{-1}(D)$ is algebraically isomorphic to $G_1/D$, it follows that $G/\pi^{-1}(D)$ is of exponent $p$ and hence not divisible. Therefore $G$ is not $d_\alpha$-extremal and so not $\alpha$-extremal.
\end{proof}

The next lemma gives some conditions equivalent to $\alpha$-singularity for $\alpha$-pseudocompact abelian groups. For $\alpha=\omega$ we find \cite[Lemma 2.5]{AGB1}, which generalized \cite[Lemma 4.1]{DGM}.

\begin{lemma}\label{alpha-singular}
Let $\alpha$ be an infinite cardinal and let $G$ be an $\alpha$-pseudocompact abelian group. Then the following conditions are equivalent:
\begin{itemize}
    \item[(a)]$G$ is $\alpha$-singular;
    \item[(b)]there exists $m\in\mathbb N_+$ such that $G[m]\in\Lambda_\alpha(G)$;
    \item[(c)]$G$ has a torsion closed $G_\alpha$-subgroup;
    \item[(d)]there exists $N\in\Lambda_\alpha(G)$ such that $N\subseteq t(G)$;
    \item[(e)]$\widetilde G$ is $\alpha$-singular.
\end{itemize}
\end{lemma}
\begin{proof}
Let $m\in\N_+$ and let $\varphi_m:G\rightarrow G$ be the continuous homomorphism defined by $\varphi_m(x)=m x$ for every $x\in G$. Then $\ker\varphi_m=G[m]$ and $\varphi_m(G)=m G$. Let $i:G/G[m]\to m G$ be the continuous isomorphism such that $i\circ\pi=\varphi_m$, where $\pi:G\to G/G[m]$ is the canonical homomorphism. 

(a)$\Rightarrow$(b) There exists $m\in\N_+$ such that $w(m G)\leq\alpha$. Then $\psi(m G)\leq\alpha$. Since $i:G/G[m]\to m G$ is a continuous isomorphism, so $\psi(G/G[m])\leq\alpha$. This implies that $G[m]$ is a $G_\alpha$-set of $G$; then $G[m]\in\Lambda_\alpha(G)$.

(b)$\Rightarrow$(a) Suppose that $G[m]\in\Lambda_\alpha(G)$. Then the quotient $G/G[m]$ has weight $\leq\alpha$, hence it is compact. By the open mapping theorem the isomorphism $i:G/G[m]\to m G$ is also open and consequently it is a topological isomorphism. Then $w(m G)\leq\alpha$, that is $G$ is $\alpha$-singular.

(b)$\Rightarrow$(c) and (c)$\Leftrightarrow$(d) are obvious.

(d)$\Rightarrow$(b) By Corollary \ref{lambda8} $N$ is $\alpha$-pseudocompact and so $N$ is bounded-torsion by \cite[Lemma 7.4]{CRob}. Therefore there exists $m\in\N_+$ such that $m N=\{0\}$. Thus $N\subseteq G[m]$ and so $G[m]\in\Lambda_\alpha(G)$ by Corollary \ref{lambda8}(a).

(a)$\Leftrightarrow$(e) It suffices to note that $w(m G)=w(m \widetilde G)$ for every $m\in\mathbb N$, because $m G$ is dense in $m \widetilde G$ for every $m\in\mathbb N$.
\end{proof}

The following lemma proves one implication of Corollary \ref{alpha-extremal_mega-theorem} and it is the generalization of \cite[Proposition 4.7]{DGM}.

\begin{lemma}\label{singular->c-extremal}
Let $\alpha$ be an infinite cardinal and let $G$ be an $\alpha$-singular $\alpha$-pseudocompact abelian group. Then $r_0(G/D)=0$ for every $G_\alpha$-dense subgroup $D$ of $G$. In particular $G$ is $c_\alpha$-extremal.
\end{lemma}
\begin{proof}
By definition there exists a positive integer $m$ such that $w(m G)\leq \alpha$. Let $D$ be a $G_\alpha$-dense subgroup of $G$. Since $m D$ is a $G_\alpha$-dense subgroup of $m G$ and $w(m G)\leq \alpha$, so $m D= m G$. Therefore $m G\leq D$ and hence the quotient $G/D$ is bounded-torsion. In particular $r_0(G/D)=0$.

Moreover $G$ is $c_\alpha$-extremal by the previous part, noting that a $G_\alpha$-dense subgroup $D$ of $G$ is dense $\alpha$-pseudocompact by Corollary \ref{alpha-psc<->G_alpha-dense}.
\end{proof}

For a product $G=\prod_{i\in I}G_i$ of topological groups, let $\Sigma_\alpha G=\{x=(x_i)\in G:|\text{supp}(x)|\leq\alpha\}$ be the \emph{$\alpha$-$\Sigma$-product} of $G$ \cite{DG1}. Then $\Sigma_\alpha G$ is $G_\alpha$-dense in $G$.

\begin{proof}[\emph{\textbf{Proof of Theorem \ref{MegaThm}}}]
(a)$\Rightarrow$(c) Suppose that there exists a continuous surjective homomorphism $\varphi$ of $K$ onto $S^\beta$, where $S$ is a metrizable compact non-torsion abelian group and $\beta>\alpha$. Then $\varphi$ is also open. Note that $S^\beta=(S^\alpha)^\beta$. Let $T=S^\alpha$. Then $D=\Sigma_\alpha T^\beta$ is a $G_\alpha$-dense subgroup of $T^\beta$, so dense $\alpha$-pseudocompact by Corollary \ref{alpha-psc<->G_alpha-dense}, and $D$ trivially intersects the diagonal subgroup $\Delta T^\beta=\{x=(x_i)\in T^\beta: x_i=x_j\ \mbox{for every}\ i,j<\beta\}$, which is topologically isomorphic to $T=S^\alpha$. Consequently $r_0(T^\beta/D)\geq 2^\alpha$. Therefore $T^\beta$ is not $c_\alpha$-extremal and $K$ is not $c_\alpha$-extremal as well by Proposition \ref{quozd}.

(b)$\Rightarrow$(a) is Lemma \ref{singular->c-extremal} and (b)$\Leftrightarrow$(c) is \cite[Theorem 1.6]{DG1}.
\end{proof}

\subsection{Proof of the main theorem}\label{proof}

\begin{proposition}\label{non-sing->big_free-rank}
Let $\alpha$ be an infinite cardinal. If $G$ is a non-$\alpha$-singular $\alpha$-pseudocompact abelian group, then $r_0(N)=r_0(G)\geq 2^\alpha$ for every $N\in\Lambda_\alpha(G)$.
\end{proposition}
\begin{proof}
First we prove that $r_0(G)\geq 2^\alpha$. Since $G$ is non-$\alpha$-singular, then $K=\widetilde G$ is non-$\alpha$-singular as well by Lemma \ref{alpha-singular}. By Theorem \ref{MegaThm} there exists a continuous surjective homomorphism $\varphi:K\to S^I$, where $S$ is a metrizable compact non-torsion abelian group and $|I|>\alpha$. Let $J$ be a subset of $I$ of cardinality $\alpha$ and consider the composition $\phi$ of $\varphi$ with the canonical projection $S^I\to S^J$. The group $S^J$ has weight $\alpha$ and free rank $2^\alpha$. Since $G$ is dense in $K$, it follows that $\phi(G)$ is dense in $S^J$. But $G$ is $\alpha$-pseudocompact and so $\phi(G)$ is compact. Therefore $\phi(G)=S^J$ and hence $r_0(G)\geq 2^\alpha$.

Let $N\in\Lambda_\alpha(G)$. Then $N$ is $\alpha$-pseudocompact by Corollary \ref{lambda8}. Since $G$ is non-$\alpha$-singular, so $N$ is non-$\alpha$-singular. In fact, by Lemma \ref{alpha-singular} if $m\in\N_+$ then $G[m]\not\in\Lambda_\alpha(G)$. By Corollary \ref{lambda8}(a) if $m\in\N_+$ then $N[m]\not\in\Lambda_\alpha(G)$, so $N[m]\not\in \Lambda_\alpha(N)$ by Corollary \ref{lambda8}(b) and hence $N$ is non-$\alpha$-singular by Lemma \ref{alpha-singular}. 
Then $r_0(N)\geq 2^\alpha$ by the first part of the proof. Moreover $r_0(G)=r_0(G/N)\cdot r_0(N)$. By Theorem \ref{lambda4} $w(G/N)\leq\alpha$ and so $G/N$ is compact with $|G/N|\leq 2^\alpha$. Hence $r_0(G)=r_0(N)$.
\end{proof}

The next lemma shows that $c_\alpha$-extremality is hereditary for $\alpha$-pseu\-do\-com\-pact subgroups that are sufficiently large. This is the generalization to $\alpha$-pseudocompact groups of \cite[Theorem 4.11]{DGM}.

\begin{lemma}\label{str_c_extr}%\emph{\cite[Theorem 4.11]{DGM}}
Let $\alpha$ be an infinite cardinal and let $G$ be a $c_\alpha$-extremal pseudocompact abelian group. Then every $\alpha$-pseudocompact subgroup of $G$ of index $\leq 2^\alpha$ is $c_\alpha$-extremal. In particular, every $N\in\Lambda_\alpha(G)$ is $c_\alpha$-extremal.
\end{lemma}
\begin{proof}
Aiming for a contradiction, assume that there exists an $\alpha$-pseu\-do\-com\-pact subgroup $N$ of $G$ with $|G/N|\leq 2^\alpha$ such that $N$ is not $c_\alpha$-extremal. Then there exists a dense $\alpha$-pseudocompact subgroup $D$ of $N$ with $r_0(N/D)\geq 2^\alpha$. Therefore $|G/N|\leq r_0(N/D)$. By \cite[Corollary 4.9]{DGM} there exists a subgroup $L$ of $G/D$ such that $L+N/D=G/D$ and $r_0((G/D)/L)\geq r_0(N/D)$. Let $\pi:G\rightarrow G/D$ be the canonical projection and $D_1=\pi^{-1}(L)$. Then $N+D_1=G$. Since $D$ is $G_{\alpha}$-dense in $N$ by Corollary \ref{alpha-psc<->G_alpha-dense}, so $\overline{D_1}^{P_\alpha G}\supseteq N+D_1=G$ and so $D_1$ is $G_{\alpha}$-dense in $G$; equivalently $D_1$ is dense $\alpha$-pseudocompact in $G$ by Corollary \ref{alpha-psc<->G_alpha-dense}. Since $G/\pi^{-1}(L)$ is algebraically isomorphic to $(G/D)/L$, it follows that $r_0(G/D_1)=r_0((G/D)/L)\geq r_0(N/D)\geq 2^\alpha$. We have produced a $G_{\alpha}$-dense subgroup $D_1$ of $G$ with $r_0(G/D_1)\geq 2^\alpha$, a contradiction.

If $N\in\Lambda_\alpha(G)$ then $w(G/N)\leq\alpha$ and so $G/N$ is compact. Hence $|G/N|\leq 2^\alpha$. So $N$ is $c_\alpha$-extremal by the previous part of the proof.
\end{proof}

\begin{lemma}\label{CvM-Lemma}\emph{\cite[Lemma 3.2]{CvM2}}
Let $\alpha$ be an infinite cardinal and suppose that $\mathcal A$ is a family of subsets of $2^\alpha$ such that:
\begin{itemize}
	\item[(1)]for $\mathcal B\subseteq\mathcal A$ such that $|\mathcal B|\leq\alpha$, $\bigcap_{B\in\mathcal B}B\in\mathcal A$ and
	\item[(2)]each element of $\mathcal A$ has cardinality $2^\alpha$.
\end{itemize}
Then there exists a countable infinite family $\mathcal B$ of subsets of $2^\alpha$ such that:
\begin{itemize}
	\item[(a)]$B_1\cap B_2=\emptyset$ for every $B_1,B_2\in\mathcal B$ and
	\item[(b)]if $A\in\mathcal A$ and $B\in\mathcal B$, then $|A\cap B|=2^\alpha$.
\end{itemize}
\end{lemma}

Now we can prove our main results.

\begin{proof}[\emph{\textbf{Proof of Corollary \ref{alpha-extremal_mega-theorem}}}]
If $G$ is $\alpha$-singular then $G$ is $c_\alpha$-extremal by Lem\-ma \ref{singular->c-extremal}.

Suppose that $G$ is $c_\alpha$-extremal and assume for a contradiction that $G$ is not $\alpha$-singular. By Theorem \ref{G>c2} $r_0(G)\leq 2^\alpha$ and by Proposition \ref{non-sing->big_free-rank} $r_0(G)\geq 2^\alpha$. Hence $r_0(G)=2^\alpha$. Let $D(G_1)=\Q^{(S)}$, with $|S|=2^\alpha$, be the divisible hull of the torsion-free quotient $G_1=G/t(G)$, where $t(G)$ denotes the subgroup of all torsion elements of $G$. Let $\pi:G\to D(G_1)$ be the composition of the canonical projection $G\to G_1$ and the inclusion map $G_1 \hookrightarrow D(G_1)$.

For a subset $A$ of $S$ let $$G(A)=\pi^{-1}\left(\Q^{(A)}\right)\text{ and }\mathcal A=\{A\subseteq S: G(A)\ \text{contains some}\ N\in\Lambda_\alpha(G)\}.$$ Then $\mathcal A$ has the property that for $\mathcal B\subseteq\mathcal A$ such that $|\mathcal B|\leq\alpha$, $\bigcap_{B\in\mathcal B}B\in\mathcal A$; and $|A|= 2^\alpha$ for all $A\in\mathcal A$, as $r_0(N)=2^\alpha$ for every $N\in \Lambda_\alpha(G)$ by Proposition \ref{non-sing->big_free-rank} and $r_0(G)=2^\alpha$. 
By Lemma \ref{CvM-Lemma} there exists a partition $\{P_n\}_{n\in\N}$ of $S$ such that $|A\cap P_n|= 2^\alpha$ for every $A\in\mathcal A$ and for every $n\in\N$. Define $V_n=G(P_0\cup\ldots\cup P_n)$ for every $n\in\N$ and note that $G=\bigcup_{n\in\N} V_n$. By Lemma \ref{cconn0} there exist $m\in\N$ and $N\in\Lambda(G)$ such that $D=V_m\cap N$ is $G_\alpha$-dense in $N$, so dense $\alpha$-pseudocompact in $N$ by Corollary \ref{alpha-psc<->G_alpha-dense}. By Lemma \ref{str_c_extr}, to get a contradiction it suffices to show that $r_0(N/D)=2^\alpha$.

Let $F$ be a torsion-free subgroup of $N$ such that $F\cap D=\{0\}$ and maximal with this property. Suppose for a contradiction that $|F|=r_0(N/D)<2^\alpha$. So $\pi(F)\subseteq \Q^{(S_1)}$ for some $S_1\subseteq S$ with $|S_1|< 2^\alpha$ and $W=P_0\cup\ldots\cup P_m\cup S_1$ has $|W\cap P_{m+1}|<2^\alpha$. Consequently $W\not\in\mathcal A$ and so $N \not\subseteq G(W)$. Note that $G/G(W)$ is torsion-free, because if $x\in G$ and $m x\in G(W)$ for some $m\in\mathbb N_+$, then $m \pi(x)=\pi(m x)\in \mathbb Q^{(W)}$, so $\pi(x)\in \mathbb Q^{(W)}$ and hence $x\in G(W)$.  Take $x\in N\setminus G(W)$. Since $G/G(W)$ is torsion-free, $\langle x\rangle\cap G(W)=\{0\}$ and $x$ has infinite order. But $D+F\subseteq G(W)$ and so $\langle x\rangle \cap (D+F)=\{0\}$, that is $(F+\langle x\rangle)\cap D=\{0\}$. This contradicts the maximality of $F$.
\end{proof}

\begin{proof}[\emph{\textbf{Proof of Theorem \ref{alpha-extremal-solution}}}]
(a)$\Rightarrow$(c) If $G$ is $\alpha$-extremal, in particular it is $c_\alpha$-extremal and so $\alpha$-singular by Corollary \ref{alpha-extremal_mega-theorem}.

Suppose that $w(G)>\alpha$. Since $G$ is $\alpha$-singular, by Lemma \ref{alpha-singular} there exists $m\in\N_+$ such that $w(m G)\leq\alpha$; in particular $m G$ is compact and so closed in $G$. Since $w(m G)\leq\alpha$ and $w(G)=w(G/m G)\cdot w(m G)$, it follows that $w(G/m G)=w(G)>\alpha$. Then $G/m G$ is not $\alpha$-extremal by Theorem \ref{torsion} and so $G$ is not $\alpha$-extremal by Proposition \ref{quozd}.

(c)$\Rightarrow$(b) is Proposition \ref{ccomp} and (b)$\Rightarrow$(a) is Theorem \ref{debestr}.
\end{proof}

The following example shows that $c_\alpha$- and $d_\alpha$-extremality cannot be equivalent conditions in Theorem \ref{alpha-extremal-solution}.
Item (a) shows that $\alpha$-singular $\alpha$-pseudocompact abelian groups need not be $d_\alpha$-extremal and also that there exists a $c_\alpha$-extremal (non-compact) $\alpha$-pseudocompact abelian group of we\-ight $>\alpha$, which is not $d_\alpha$-extremal. It is the analogous of \cite[Example 4.4]{DGM}. In item (b) we give an example of a non-$c_\alpha$-extremal $d_\alpha$-extremal $\alpha$-pseu\-do\-com\-pact abelian group of weight $>\alpha$.

\begin{example}\label{Example-cd}
Let $\alpha$ be an infinite cardinal.
\begin{itemize}
	\item[(a)]Let $p\in\Prm$ and let $H$ be the subgroup of $\Z(p)^{2^\alpha}$ defined by $H=\Sigma_\alpha\Z(p)^{2^\alpha}$. Then $H$ is $\alpha$-pseudocompact by Corollary \ref{alpha-psc<->G_alpha-dense}, because it is $G_\alpha$-dense in $\Z(p)^{2^\alpha}$. The group $G=\T^\alpha\times H$ is an $\alpha$-singular (so $c_\alpha$-extremal by Lemma \ref{singular->c-extremal}) $\alpha$-pseudocompact abelian group with $r_0(G)=2^\alpha$ and $w(G)=2^\alpha>\alpha$. Thus $G$ is not $d_\alpha$-extremal by Theorem \ref{alpha-extremal-solution}.
	\item[(b)]Let $G=\T^{2^\alpha}$. Then $G$ is an $\alpha$-pseudocompact divisible abelian group $G$ of weight $>\alpha$. So $G$ is $d_\alpha$-extremal of weight $>\alpha$. We can prove that $G$ is not $c_\alpha$-extremal as in the proof of Theorem \ref{MegaThm}, that is, noting that $G=(\T^\alpha)^{2^\alpha}=T^{2^\alpha}$ and $\Sigma_\alpha T^{2^\alpha}$ is a $G_\alpha$-dense subgroup of $T^{2^\alpha}$ such that $r_0(T^{2^\alpha}/\Sigma T^{2^\alpha})\geq 2^\alpha$. That $G$ is not $c_\alpha$-extremal follows also from Theorem \ref{G>c2}, because $G$ has free rank $> 2^\alpha$.
\end{itemize}
\end{example}

\subsection*{Acknowledgments}

I am grateful to Professor Dikran Dikranjan for his helpful comments and suggestions. I thank also Professor Hans Weber and the referee for their observations and suggestions.

\noindent Author: Anna Giordano Bruno\\
Address: Dipartimento di Matematica e Informatica, Universit\`a di Udine, Via delle Scienze, 206 - 33100 Udine, Italy \\
E-mail address: {\tt anna.giordanobruno@dimi.uniud.it}

\end{document}